\documentclass{article}
\usepackage{amsmath,amsfonts,stmaryrd,amssymb} 

\usepackage{enumerate} 

\usepackage[ruled]{algorithm2e} 

\usepackage[framemethod=tikz]{mdframed} 

\usepackage{listings} 
\usepackage{geometry} 

\usepackage[utf8]{inputenc} 
\usepackage[T1]{fontenc} 

\usepackage{XCharter} 

\mdfdefinestyle{commandline}{
	leftmargin=10pt,
	rightmargin=10pt,
	innerleftmargin=15pt,
	middlelinecolor=black!50!white,
	middlelinewidth=3pt,
	frametitlerule=false,
	backgroundcolor=black!5!white,
	frametitle={Command Line},
	frametitlefont={\normalfont\sffamily\color{white}\hspace{-1em}},
	frametitlebackgroundcolor=black!50!white,
	nobreak,
}

\mdfdefinestyle{file}{
	innertopmargin=4.0\baselineskip,
	innerbottommargin=3.5\baselineskip,
	topline=false, bottomline=false,
	leftline=false, rightline=false,
	leftmargin=2cm,
	rightmargin=2cm,
	singleextra={%
		\draw[fill=black!10!white](P)++(0,-1.2em)rectangle(P-|O);
		\node[anchor=north west]
		at(P-|O){\ttfamily\mdfilename};
		\def\l{3em}
		\draw(O-|P)++(-\l,0)--++(\l,\l)--(P)--(P-|O)--(O)--cycle;
		\draw(O-|P)++(-\l,0)--++(0,\l)--++(\l,0);
	},
	nobreak,
}

\mdfdefinestyle{question}{
	innertopmargin=3.5\baselineskip,
	innerbottommargin=3.0\baselineskip,
	roundcorner=5pt,
	nobreak,
	singleextra={%
		\draw(P-|O)node[xshift=1em,anchor=west,fill=white,draw,rounded corners=5pt]{%
		Question \theQuestion\questionTitle};
	},
}

\newcounter{Question} 

\mdfdefinestyle{warning}{
	topline=false, bottomline=false,
	leftline=false, rightline=false,
	nobreak,
	singleextra={%
		\draw(P-|O)++(-0.5em,0)node(tmp1){};
		\draw(P-|O)++(0.5em,0)node(tmp2){};
		\fill[black,rotate around={45:(P-|O)}](tmp1)rectangle(tmp2);
		\node at(P-|O){\color{white}\scriptsize\bf !};
		\draw[very thick](P-|O)++(0,-1em)--(O);
	}
}

\mdfdefinestyle{info}{%
	topline=false, bottomline=false,
	leftline=false, rightline=false,
	nobreak,
	singleextra={%
		\fill[black](P-|O)circle[radius=0.4em];
		\node at(P-|O){\color{white}\scriptsize\bf i};
		\draw[very thick](P-|O)++(0,-0.8em)--(O);
	}
}

\usepackage[mathlines,pagewise]{lineno}
\usepackage{indentfirst}
\usepackage{amsmath}
\usepackage[utf8]{inputenc}
\usepackage[english]{babel}
\usepackage{amsthm}
\usepackage{bbm}
\usepackage{dsfont}
\usepackage{blindtext}
\usepackage{amssymb}
\usepackage{extarrows}
\usepackage{authblk}
\newtheorem{theorem}{Theorem}[section]
\newtheorem{corollary}{Corollary}[theorem]
\newtheorem{lemma}[theorem]{Lemma}
\theoremstyle{remark}
\newtheorem*{remark}{Remark}
\theoremstyle{definition}
\newtheorem{definition}{Definition}[section]

\title{Total variation distance between a jump-equation and its Gaussian approximation}
\author[*]{Vlad Bally}
\author[*]{Yifeng Qin}

\affil[*]{Universit\'{e} Gustave Eiffel, LAMA (UMR CNRS, UPEMLV, UPEC), MathRisk INRIA, F-77454 Marne-la-Vall\'{e}e, France.
Email address: bally@univ-mlv.fr}
\date{}

\renewcommand{\ell}{l}

\begin{document}


\maketitle

\textbf{Abstract}

We deal with stochastic differential equations with jumps. In order to obtain an accurate approximation scheme, it is usual to replace the "small jumps" by a Brownian motion. In this paper, we prove that for every fixed time $t$, the approximate random variable $X^\varepsilon_t$ converges to the original random variable $X_t$ in total variation distance and we estimate the error. We also give an estimate of the distance between the densities of the laws of the two random variables. These are done by using some integration by parts techniques in Malliavin calculus.


\tableofcontents

\section{Introduction}
In this paper we consider the stochastic differential equation with jumps%
\[
X_{t}=x+\int_{0}^{t}\int_{(0,1]}c(s,z,X_{s-})N_{\mu }(ds,dz)
\]%
where $N_{\mu }$ is a Poisson random measure on $(0,1]$ with compensator $\mu (dz)ds$ 
and $c$ is a coefficient which verifies strong regularity
hypotheses (see \textbf{Hypotheses 2.1-2.4} in Section 2.1). 
The typical example that we have in mind is $\mu (dz)=\frac{dz}{z^{1+\rho}}1_{\{z\in (0,1]\}}$, 
with $\rho\in[0,1),$ so this is a truncated stable process - however, throughout the paper, we keep the general framework in which $\mu $ is a
measure which has infinite mass around zero. 
Our aim is to replace the "small jumps" by a space-time Brownian motion:%
\begin{eqnarray}
X_{t}^{\varepsilon } &=&x+\int_{0}^{t}\int_{\{z>\varepsilon
\}}c(s,z,X_{s-}^{\varepsilon })N_{\mu }(ds,dz)  \label{i1} \\
&&+\int_{0}^{t}b_{\varepsilon }(s,X_{s}^{\varepsilon
})ds+\int_{0}^{t}\int_{(0,\varepsilon ]}c(s,z,X_{s}^{\varepsilon })W_{\mu
}(ds,dz),  \nonumber
\end{eqnarray}%
where $W_{\mu }(ds,dz)$ is a space-time Brownian motion (in the sense of
Walsh $\cite{ref15}$) with covariance $\mu (dz)ds$, $x\in\mathbb{R}$, and the coefficient $b_{\varepsilon}$ is defined by%
\[
b_{\varepsilon }(s,x)=\int_{(0,\varepsilon ]}c(s,z,x)\mu (dz).
\]%
The interest of such approximations appears in various frameworks.

Our main motivation comes from numerical computations. 
If $\mu (E)<\infty$ then there are a finite number of jumps in any compact interval of time, so $X_{t}$
may be represented by means of a compound Poisson process which may be
explicitly simulated. But if $\mu (E)=\infty $ this is not possible anymore
(except in very particular situations - see Talay and Protter $\cite{ref30}$ for example), and the "small jumps" should be truncated to revert to the case of a finite measure. 
This procedure is rather rough and gives large errors. 
In order to improve the approximation scheme, one may replace the "small jumps", namely those smaller than $\varepsilon$, 
by a stochastic integral with respect to $W_\mu(ds,dz)$. 
Note that the Poisson measure $dN_{\mu }$ is not compensated, which is why the drift corresponding to $b_{\varepsilon }$ appears. 
This idea goes back to Asmussen and Rosinski $\cite{ref23}$. 
In the case of $SDE^{\prime }s$ driven by a L\'{e}vy process, Fournier $\cite{ref31}$ gives a precise estimate of the
error and compares the approximation obtained just by truncating the small jumps 
to the one obtained by adding a Gaussian noise as in (\ref{i1}). 
An enlightening discussion on the complexity of the two methods is also given. 
However, in that paper, the strong error is considered, while in our paper we discuss the weak error.

A second motivation comes from modelization problems in chemistry and
biology: we are concerned by reactions which are naturally modelled by means
of jump processes containing two regimes: one is very rapid but
the jumps are small, and another is much slower and the jumps are
larger -- see for example $\cite{ref21}$, $\cite{ref22}$, $\cite{ref24}$, $\cite{ref25}$, $\cite{ref28}$, $\cite{ref29}$.
In this case the regime corresponding to the rapid scale may be modelled by a stochastic
integral with respect to a Gaussian process and the slow regime by a compound Poisson process. It may also be reasonable to consider an
intermediary regime and this would be modelled by a drift term.

A third motivation is given by a class of statistical problems (see $\cite{ref16}$, $\cite{ref26}$ and references therein), 
where a stochastic process is observed at various times and it should be decided whether its increments are due to small jumps or to a Gaussian component.
In this framework it is important to estimate the error in total
variation sense. The authors explain that, if the error in total variation
between the laws of $X_{t}$ and of $X_{t}^{\varepsilon }$ goes to zero, then
there is no way to construct a test which decides if the noise comes from
small jumps or from the Brownian motion. 
So, asymptotically, the two models contain the same information.

Let us now discuss briefly our results and the relation to previously available estimates. 
If $L_t$ (respectively $L^{\varepsilon }_t$) represents the infinitesimal operator of $X_t$
(respectively of $X^{\varepsilon }_t)$ then a development in Taylor series of
order two gives 
\[
\left\Vert (L_t-L^{\varepsilon }_t)f\right\Vert _{\infty }\leq C\left\Vert
f\right\Vert _{3,\infty }\int_{(0,\varepsilon
]}\left\vert \hat{c}(z)\right\vert ^{3}\mu (dz),
\]%
where $\hat{c}(z):=\sup\limits_{s\leq T}\sup\limits_{x\in \mathbb{R}}\left\vert
c(s,z,x)\right\vert$ and  $\left\Vert f \right\Vert _{3,\infty }:=\sum\limits_{0\leq i\leq 3}\Vert f^{(i)}\Vert_\infty$. Then a Trotter-Kato type argument yields
\[
\sup_{s\leq T}\left\Vert (P_{t}-P_{t}^{\varepsilon })f\right\Vert _{\infty
}\leq C\left\Vert f\right\Vert _{3,\infty
}\int_{(0,\varepsilon ]}\left\vert \hat{c}(z)\right\vert ^{3}\mu (dz)
\]%
where $P_{t}$ (respectively of $P_{t}^{\varepsilon })$ represents the semigroup of $X_t$ (respectively of $X^{\varepsilon }_t).$

The drawback of the above estimate is that the bound on the error involves $\left\Vert f\right\Vert _{3,\infty }$, 
so it only applies to smooth test functions. 
The main contribution of our paper is to replace $\left\Vert f\right\Vert _{3,\infty }$ by $\left\Vert f\right\Vert_{\infty }$, 
so as to prove convergence in total variation distance. 
This is done under non-degeneracy and regularity assumptions on the coefficient $c.$
Moreover, under these hypotheses, we prove that $\mathbb{P}(X_{t}(x)\in dy)=p_{t}(x,y)dy$ and 
$\mathbb{P}(X_{t}^\varepsilon(x)\in dy)=p_{t}^{\varepsilon }(x,y)dy$ with smooth densities $%
y\mapsto p_{t}(x,y)$ and $y\mapsto p_{t}^{\varepsilon }(x,y).$ And,
for every $k$ and every $\delta>0$, we obtain%
\[
\left\Vert \partial _{y}^{k}p_{t}-\partial _{y}^{k}p_{t}^{\varepsilon}\right\Vert_{\infty }
\leq C_{k,\delta}\Big(\int_{(0,\varepsilon ]}\left\vert \hat{c}(z)\right\vert ^{3}\mu (dz)\Big)^{1-\delta}.
\]%
This proves that $p_t^\varepsilon$ converges to $p_t$ in distribution norms as $\epsilon \to 0$. 

Our approach uses a strategy based on integration by parts (an abstract Malliavin calculus) developed in $\cite{ref7}$.

The paper is organized as follows. In Section 2, we give the main results and in Section 3, we recall the integration by parts technique introduced in $\cite{ref7}$ and used here. In Section 4, we use these results in the framework of stochastic equations with jumps and we prove the main result (\textbf{Theorem 2.2}). The Appendix contains technical estimates concerning Sobolev norms in Malliavin sense.

\bigskip

\section{Main results}
\subsection{The basic equation and the hypotheses}
A time horizon $T>0$ will be fixed throughout the paper. As already mentioned, we deal with the one-dimensional jump equation 
\begin{equation}
X_{t}=x+\int_{0}^{t}\int_{(0,1]}c(s,z,X_{s-})N_{\mu}(ds,dz),  \label{1.1}
\end{equation}%
where $N_{\mu}$ is a Poisson point measure with intensity 
$\widehat{N_{\mu} }(ds,dz)=\mu (dz)ds$,
and $\mu$ is a positive $\sigma$-finite measure on $(0,1]$, $t\in[0,T]$.

For technical reasons which will be discussed in Section 4, we introduce the following change of variables. Let $\theta:(0,1]\rightarrow[1,\infty)$ be the function defined by $\theta(z)=\frac{1}{z}$, and let $\nu(dz):= \mu\circ \theta^{-1}(dz)$. Then $\nu$ is a positive $\sigma$-finite measure on $[1,\infty)$.
Consider a Poisson point measure $N_\nu(ds,dz)$ with intensity $\widehat{N_\nu}(ds,dz)=\nu(dz)ds$. 
One may then check that for every $t\in[0,T]$, $X_t$ has the same law as $\widehat{X}_t$, with $(\widehat{X}_t)_{t\in[0,T]}$ the solution of
\begin{eqnarray}
\widehat{X}_t=x+\int_{0}^{t}\int_{[1,\infty)}\widetilde{c}(s,z,\widehat{X}_{s-})N_{\nu }(ds,dz),  \label{hatX}
\end{eqnarray}
where $\widetilde{c}(s,z,x):=c(s,\frac{1}{z},x)$.

Since this paper deals with the laws of the solution to \eqref{1.1},
it is equivalent to consider the equation (\ref{hatX}). 
We formulate our hypotheses in terms of $\widetilde{c}$ and $\nu$ (instead of $c$ and $\mu$).
\medskip

\noindent \textbf{Hypothesis 2.1} (\textbf{Regularity with parameter $q^\ast$})
The map $(s,z)\mapsto \widetilde{c}(s,z,x)$ is continuous, and
there exists a non-negative and decreasing function $\bar{c}:[1,\infty)\rightarrow\mathbb{R}_{+}$ and a constant $q^\ast \in \mathbb N$ such that for every indices $\beta_1,\beta_2$, with $\beta_1\leq q^\ast$ and $\beta_2\leq q^\ast$, we have
\[
\sup_{s\in[0,T]}\sup_{x\in\mathbb{R}}(\vert \widetilde{c}(s,z,x)\vert+\vert \partial_z^{\beta_2}\partial_x^{\beta_1} \widetilde{c}(s,z,x)\vert)\leq\bar{c}(z), \quad\forall z\in[1,\infty),
\]
with 
\begin{eqnarray}
	\int_{[1,\infty)}\vert \bar{c}(z)\vert^p\nu(dz)=:\bar{c}_p<\infty,\quad\forall p\geq1. \label{cbar}
\end{eqnarray}

\begin{remark}
We will use several times the following consequence of (\ref{cbar}) and of Burkholder's inequality (see for example the Theorem 2.11 in $\cite{ref20}$): We assume that $\Phi(s,z,\omega)$ and $\varphi(s,\omega)$ are two non-negative functions such that 
\[
\vert\Phi(s,z,\omega)\vert\leq\bar{c}(z)\varphi(s,\omega).
\]
Then for any $p\geq2$,
\begin{eqnarray}
	\mathbb{E}\Big\vert\int_0^T\int_{[1,\infty)}\Phi(s,z,\omega)N_\nu(ds,dz)\Big\vert^p
	\leq C\times C_p\, \mathbb{E}\int_0^T\vert\varphi(s,\omega)\vert^p ds, \label{Burk}
\end{eqnarray}
where $C_p=max\{(\bar{c}_2)^{\frac{p}{2}},\bar{c}_p,(\bar{c}_1)^p\}$ and $C$ is a constant depending on $p$ and $T$.
\end{remark}

\begin{proof}
By compensating $N_\nu$, using Burkholder's inequality and (\ref{cbar}), we have
\begin{eqnarray*}
&&\mathbb{E}\vert\int_0^T\int_{[1,\infty)}\Phi(s,z,\omega)N_\nu(ds,dz)\vert^p\leq C(\mathbb{E}\vert\int_0^T\int_{[1,\infty)}\Phi(s,z,\omega)\widetilde{N_\nu}(ds,dz)\vert^p+\mathbb{E}\vert\int_0^T\int_{[1,\infty)}\Phi(s,z,\omega)\nu(dz)ds\vert^p)\\
&&\leq C(\mathbb{E}\int_0^T(\int_{[1,\infty)}\vert\Phi(s,z,\omega)\vert^2\nu(dz))^{\frac{p}{2}}ds+\mathbb{E}\int_0^T\int_{[1,\infty)}\vert\Phi(s,z,\omega)\vert^p\nu(dz)ds+\mathbb{E}\int_0^T\vert\int_{[1,\infty)}\Phi(s,z,\omega)\nu(dz)\vert^{p}ds)\\
&&\leq C\times C_p\mathbb{E}\int_0^T\vert\varphi(s,\omega)\vert^pds.
\end{eqnarray*}
\end{proof}\medskip 

\noindent \textbf{Hypothesis 2.2}
There exists a non-negative function $\breve{c}:[1,\infty)\rightarrow\mathbb{R}_{+}$ such that 
$\int_{[1,\infty)}\vert \breve{c}(z)\vert^p\nu(dz)=:\breve{c}_p<\infty,\quad\forall p\geq1$, and
\[
\Big\vert \frac{\partial_{x}\widetilde{c}(s,z,x)}{1+\partial_{x}\widetilde{c}(s,z,x)}\Big\vert\leq\breve{c}(z),\quad\forall s\in[0,T], x\in\mathbb{R}, z\in[1,\infty).
\]
To avoid overburdening notation, since both hypotheses 2.1 and 2.2 apply, we will take $\breve{c}(z)=\bar{c}(z)$.
\medskip

\noindent\textbf{Hypothesis 2.3} (\textbf{Ellipticity})
There exists a non-negative function $\underline{c}:[1,\infty)\rightarrow\mathbb{R}_{+}$ such that for every $s\in[0,T], x\in\mathbb{R}, z\in[1,\infty)$,
\[
\left\vert \partial _{z}\widetilde{c}(s,z,x)\right\vert ^{2}\geq \underline{c}%
(z)\quad and\quad \left\vert \widetilde{c}(s,z,x)\right\vert ^{2}\geq \underline{c}%
(z).
\]\medskip

\noindent \textbf{Hypothesis 2.4} (\textbf{Sector condition})
This is a supplementary hypothesis concerning the measure $\nu$. Two version of this hypothesis will be used; we state them separately below. 
Let $I_k=[k,k+1), k\in\mathbb{N}$ and $m_{k}=\nu(I_{k})$.
\smallskip 

\noindent $(a)$ \textbf{Strong sector condition}:
We say that the strong sector condition is satisfied if
there exist constants $\varepsilon_{\ast}>0$ and $\alpha_1\geq \alpha_0 >\alpha_2>0$, such that
\begin{eqnarray}
\mathbbm{1}_{I_k}(z)\frac{\nu (dz)}{m_k}&\geq&  \mathbbm{1}_{I_k}(z)\frac{\varepsilon_{\ast}}{ z^{1-\alpha_1 }}dz  \quad\ \ \text{for all $k\in\mathbb{N}$},  \label{cm9}\\
\underline{c}(z)&\geq& e^{- z^{\alpha_2}}\qquad\quad\  \qquad\text{ for all $z \geq 1$ and,}\nonumber  \\
\int_1^\infty\frac{\vert\bar{c}(z)\vert^p}{z^{1-\alpha_0}}dz &<& \infty\qquad\qquad\qquad\quad \text{for all $p\geq1$}. \label{aLfault}
\end{eqnarray}
Notice that if (\ref{cm9}) is true for some $\alpha_1$, then it is also true for any ${\alpha}\leq\alpha_1$. So (\ref{cm9}) also implies
\begin{eqnarray*}
\mathbbm{1}_{I_k}(z)\frac{\nu (dz)}{m_k}\geq  \mathbbm{1}_{I_k}(z)\varepsilon_kdz,\quad with\quad
\varepsilon_k=\frac{\varepsilon_\ast}{(k+1)^{1-{\alpha}}},\label{cm9bit}
\end{eqnarray*}for any ${\alpha}\leq\alpha_1$.


\smallskip 

\noindent $(b)$ \textbf{Weak sector condition}:
We say that the weak sector condition holds if there exist constants $\varepsilon_{\ast}>0$ and $\alpha>0$, such that for every $k\in\mathbb{N}$, we have
\begin{eqnarray}
\mathbbm{1}_{I_k}(z)\frac{\nu (dz)}{m_k}&\geq&  \mathbbm{1}_{I_k}(z)\frac{\varepsilon_{\ast}}{ z}dz  \qquad\quad \text{for all $k\in\mathbb{N}$},  \label{cm10}\\
\underline{c}(z)&\geq&  \frac{1}{z^{\alpha }} \qquad\qquad\  \qquad\text{ for all $z \geq 1$ and},\nonumber \\
\int_1^\infty\frac{\vert\bar{c}(z)\vert^p}{z}dz&<&\infty\qquad\qquad\qquad\ \ \text{ for all $p\geq1$}. \label{bLfault}
\end{eqnarray}
We notice that (\ref{cm10}) also implies $\mathbbm{1}_{I_k}(z)\frac{\nu (dz)}{m_k}\geq  \mathbbm{1}_{I_k}(z)\varepsilon_kdz$ with
$\varepsilon_k=\frac{\varepsilon_\ast}{k+1}.$
\medskip

\begin{remark}
	Notice that hypotheses 2.1, 2.2 and 2.3 are analogous to those in $\cite{ref10}$ $(2-7,2-26,2-24)$.
\end{remark}
\smallskip 

{\bf Henceforth, we will suppose that hypotheses 2.1-2.3 hold, as well as either 2.4$(a)$ or 2.4$(b)$. }

\subsection{Approximation}
We come back now to equation ($\ref{1.1}$). The goal of this paper is to replace the small jumps in ($\ref{1.1}$) by a drift and a Brownian motion. In equation ($\ref{1.1}$), the Poisson point measure $N_\mu$ is not compensated, so the first step is to introduce a drift (see $b_\varepsilon$ below) which represents the compensator. Afterwards, we introduce a space-time Brownian motion $W_\mu$ in order to replace the "compensated small jumps": 
\begin{eqnarray}
X_{t}^{\varepsilon } &=&x+\int_{0}^{t}\int_{\{z >
\varepsilon \}}c(s,z,X_{s-}^{\varepsilon })N_{\mu
}(ds,dz)   \nonumber\\
&&+\int_{0}^{t}b_{\varepsilon }(s,X_{s}^{\varepsilon
})ds+\int_{0}^{t}\int_{(0,\varepsilon]}c(s,z,X_s^{\varepsilon})W_\mu(dz,ds),  \label{1.4}
\end{eqnarray}%
where%
\begin{eqnarray*}
b_{\varepsilon }(s,x)=\int_{(0,\varepsilon]
}c(s,z,x)\mu (dz)
\end{eqnarray*}%
and $W_{\mu }$ is a space-time Brownian motion with covariance measure $\mu(dz)ds$, which is independent of $N_\mu$.

Let us discuss this equation. We notice that we keep the
"big jumps" with $z > \varepsilon $ but we
eliminate the "small jumps" with $z \leq\varepsilon .$
We replace the "small jumps" by the drift with coefficient $b_{\varepsilon }$
and by the stochastic integral with coefficient $c.$ 
This stochastic integral is driven by the so called space-time Brownian motion $%
W_{\mu }$, as
introduced by Walsh in $\cite{ref15}$. The existence and uniqueness of the solution to this equation ($\ref{1.4}$) are also given by Kunita (see $\cite{ref20},\cite{ref13}$). 

We recall that we work on a fixed interval of time $[0,T]$. We now precise the filtration that we consider. Let
\begin{eqnarray}
&\mathcal{F}_t^W=\sigma(W_\mu(\varphi\mathbbm{1}_{[0,t]}):\  \varphi\in L^2((0,1]\times[0,T],\mu\times Leb)),  \nonumber\\
&\mathcal{F}_t^N=\sigma(N_\mu(\varphi\mathbbm{1}_{[0,t]}):\  \varphi\in L^1((0,1]\times[0,T],\mu\times Leb)),  \nonumber\\
&\mathcal{F}_t=\mathcal{F}_t^W\bigvee\mathcal{F}_t^N, \label{Ft}
\end{eqnarray}
where $Leb$ denotes the Lebesgue measure and
\[
W_\mu(\varphi)=\int_0^T\int_{(0,1]}\varphi(s,z)W_\mu(ds,dz),\quad N_\mu(\varphi)=\int_0^T\int_{(0,1]}\varphi(s,z)N_\mu(ds,dz).
\]
So, $X_t^\varepsilon$ is $\mathcal{F}_t-$measurable and $X_t$ is $\mathcal{F}_t^N-$measurable.

We denote 
\begin{eqnarray}
L^2(W)=\{F\in\mathcal{F}_T^W:\mathbb{E}\vert F\vert^2<\infty\},\quad L^2(N)=\{G\in\mathcal{F}_T^N:\mathbb{E}\vert G\vert^2<\infty\}.  \label{ap}
\end{eqnarray}

\begin{remark}
    Let $\Phi$ be an adapted and piecewise constant process, that is 
    \[
    \Phi(s,z,\omega)=\sum_{i=1}^n\sum_{j=1}^m\Phi_{i,j}(\omega)\mathbbm{1}_{[s_i,s_{i+1})}(s)\mathbbm{1}_{B_j}(z),
    \]where $0\leq s_1<\cdots<s_n, i=1,\cdots,n,$ $B_j\in\mathcal{B}((0,1]), j=1,\cdots,m,$ are disjoint sets. Suppose that $\Phi_{i,j}$ are $\mathcal{F}^W_{s_i}-$measurable for all $j=1,\cdots,m$, and $\sup\limits_{i,j}\mathbb{E}\vert\Phi_{i,j}\vert^2<\infty.$
    Then for every $G\in L^2(N)$, we have
    \begin{eqnarray}
    \mathbb{E}\Big[G\times\int_{0}^{T}\int_{(0,1] } \Phi(s,z,\omega)W_{\mu }(ds,dz)\Big]=0.  \label{LWN}
    \end{eqnarray}
\end{remark}
\begin{proof}
Since $W_\mu([s_i,s_{i+1})\times B_j)$ is centered and independent of $\Phi_{i,j}$ and of $G$, for all $i=1,\cdots,n,\ j=1,\cdots,m$, it follows that $\mathbb{E}[G\Phi_{i,j}W_\mu([s_i,s_{i+1})\times B_j)]=0$. Then it extends by linearity, and so (\ref{LWN}) is true.
\end{proof}

Now we write the infinitesimal operator of $X_s$ and $X_s^{\varepsilon}$, respectively: For $\phi\in C_b^3(\mathbb{R})$ (the space of functions with continuous and bounded derivatives up to order 3),
\begin{eqnarray}
L_{s}\phi (x)&=&\int_{(0,1]}(\phi (x+c(s,z,x))-\phi (x))\mu (dz) \qquad \text{ and}\nonumber\\
L_{s}^{\varepsilon }\phi (x) &=&\int_{\{z >\varepsilon \}}(\phi (x+c(s,z,x))-\phi (x))\mu (dz) \label{1.2} 
+ \phi' (x)b_{\varepsilon }(s,x) +\tfrac{1}{2}\phi'' (x)a_{\varepsilon}(s,x), 
\end{eqnarray}%
where 
\[
a_{\varepsilon}(s,x)=\int_{(0,\varepsilon]}\vert c(s,z,x)\vert^2\mu(dz).
\]
Using Taylor's formula of order 2, we find  
\begin{eqnarray*}
L_{s}\phi (x) =\int_{\{z >\varepsilon \}}(\phi (x+c(s,z,x))-\phi (x))\mu (dz)  +\phi' (x)b_{\varepsilon }(s,x) +\tfrac{1}{2}\phi'' (x)a_{\varepsilon}(s,x) +R_s(x),
\end{eqnarray*}%
where \begin{eqnarray*}
\vert R_s(x)\vert &\leq&\tfrac{1}{6}\Vert\phi\Vert_{3,\infty}\int_{(0,\varepsilon]}\vert c(s,z,x)\vert^3\mu (dz),
\end{eqnarray*}
with $\left\Vert \phi \right\Vert _{\ell,\infty }:=\sum\limits_{0\leq i\leq \ell}\Vert \phi^{(i)}\Vert_\infty$, the sum of all the uniform norms of the derivatives of function $\phi$ up to order~$\ell$. In conclusion, we find  
\begin{eqnarray}
\left\Vert (L_{s}-L_{s}^{\varepsilon })\phi\right\Vert_\infty
=\left\Vert R_s\right\Vert_\infty
\leq\tfrac{1}{6}\left\Vert \phi \right\Vert _{3,\infty }\eta _{3}(\varepsilon ), \label{3.00}
\end{eqnarray}
with 
\begin{equation}
    \eta _{p}(\varepsilon )
    =\int_{(0,\varepsilon]}\left\vert \bar{c}(1/z)\right\vert ^{p}\mu (dz)
    =\int_{[\varepsilon^{-1},\infty)}\left\vert \bar{c}(z)\right\vert ^{p}\nu (dz),
    \quad p\geq1. \label{1.5}
\end{equation}
Then, we can give an estimate of the distance between the semigroups. We use the standard semigroup notation, which we remind below. Let $[X_t(s,x)]_{t\geq s}$ and $[X_t^\varepsilon(s,x)]_{t\geq s}$ be the solutions to ($\ref{1.1}$) and ($\ref{1.4}$), respectively, starting at time $s$ from point $x$. Denote by $P_{s,t}\phi(x)=\mathbb{E}\phi(X_t(s,x))$ and $P_{s,t}^\varepsilon\phi(x)=\mathbb{E}\phi(X_t^\varepsilon(s,x))$. Also, set $P_{t}:=P_{0,t}$ and $P_{t}^{\varepsilon }:=P_{0,t}^\varepsilon$. 

\begin{lemma}
	There exists a constant $C$ depending on $T$ such that for $\phi\in C_b^3(\mathbb{R})$ and $0\leq t\leq T$, we have
	\begin{equation}
		\left\Vert P_{t}\phi-P_{t}^{\varepsilon }\phi \right\Vert _{\infty }\leq
		C\left\Vert \phi \right\Vert _{3,\infty }\eta _{3}(\varepsilon ). \label{1.21}
	\end{equation}
\end{lemma}

\begin{proof}
\textbf{Step 1 Trotter-Kato method: }We know from Kunita $\cite{ref13}$ (Theorem 4.5.1) that we have the Kolmogorov forward and backward equations:
\begin{eqnarray}
&&\partial_tP_{s,t}\phi(x)=P_{s,t}L_t\phi(x),\quad \quad \partial_tP_{s,t}^\varepsilon\phi(x)=P_{s,t}^\varepsilon L_t^\varepsilon\phi(x);  \label{special}\\
&&\partial_sP_{s,t}\phi(x)=-L_sP_{s,t}\phi(x),\quad  \partial_sP_{s,t}^\varepsilon\phi(x)=-L_s^\varepsilon P_{s,t}^\varepsilon\phi(x).  \label{special*}
\end{eqnarray}

Then using Newton-Leibniz's formula and (\ref{special}), (\ref{special*}),%
\[
P_{t}^{\varepsilon }\phi (x)-P_{t}\phi (x)=\int_{0}^{t}\partial
_{s}(P_{0,s}^{\varepsilon }P_{s,t})\phi
(x)ds=\int_{0}^{t}(P_{0,s}^{\varepsilon
}(L_{s}^{\varepsilon }-L_{s})P_{s,t})\phi (x)ds.
\]%
It follows that
\begin{eqnarray}
\left\Vert P_{t}\phi-P_{t}^{\varepsilon }\phi \right\Vert _{\infty } &\leq
&\int_{0}^{t}\left\Vert P_{0,s}^{\varepsilon
}(L_{s}^{\varepsilon }-L_{s})P_{s,t}\phi \right\Vert _{\infty }ds \nonumber\\
&\leq &\int_{0}^{t}\left\Vert (L_{s}^{\varepsilon }-L_{s})P_{s,t}\phi \right\Vert _{\infty }ds \nonumber\\
&\leq &\tfrac{1}{6} \,\eta _{3}(\varepsilon )\int_{0}^{t}\left\Vert P_{s,t}\phi \right\Vert
_{3,\infty }ds.   \label{3.01}
\end{eqnarray}%

\textbf{Step 2 (propagation of regularity) }
In $\cite{ref13}$, Kunita has shown in Theorem 3.4.1 and Theorem 3.4.2 the regularity of the flow associated with the jump-diffusion. So in our case, we have 
\begin{eqnarray}
\left\Vert P_{s,t}\phi \right\Vert _{3,\infty }&\leq&\sup\limits_{x\in\mathbb{R}}(\mathbb{E}\vert\phi(X_t(s,x))\vert+\mathbb{E}\vert\partial_x\phi(X_t(s,x))\vert+\mathbb{E}\vert\partial_x^2\phi(X_t(s,x))\vert+\mathbb{E}\vert\partial_x^3\phi(X_t(s,x))\vert)   \nonumber\\
&\leq&\left\Vert
\phi \right\Vert _{3,\infty }\sup\limits_{x\in\mathbb{R}}\mathbb{E}[1+3\vert\partial_xX_{t}(s,x)\vert+3\vert\partial_x^2X_{t}(s,x)\vert+\vert\partial_x^3X_{t}(s,x)\vert]\leq {C}\left\Vert
\phi \right\Vert _{3,\infty }.  \label{1.6}
\end{eqnarray}
Substituting (\ref{1.6}) into (\ref{3.01}), we obtain (\ref{1.21}).
\end{proof}

\begin{remark}
A similar result has been obtained in $\cite{ref32}$ (Theorem 4.7).
Besides, one may also consider an approximate equation obtained just by
discarding the small jumps: 
\[
\widetilde{X_{t}^{\varepsilon }}=x+\int_{0}^{t}\int_{\{z >
\varepsilon \}}c(s,z,\widetilde{X_{s-}^{\varepsilon }})N_{\mu
}(ds,dz). 
\]
Then, if $\widetilde{L_s^\varepsilon}$ is the infinitesimal operator of $\widetilde{X_{s}^{\varepsilon }}$, we have $\Vert(L_s-\widetilde{L_s^\varepsilon})\phi\Vert_\infty\leq\Vert\phi\Vert_{1,\infty}\eta_1(\varepsilon)$.
So the same reasoning as above gives
\begin{equation}
\big\Vert P_{t}\phi -\widetilde{P_{t}^{\varepsilon }}\phi \big\Vert _{\infty }\leq
C\left\Vert \phi \right\Vert _{1,\infty }\times \eta _{1}(\varepsilon
)\rightarrow 0.  \label{1.23}
\end{equation}%

The gain in (\ref{1.21}) is that we have $\eta _{3}(\varepsilon )$
instead of $\eta_{1}(\varepsilon )$ in (\ref{1.23}), which means that we have a faster speed of convergence.
\end{remark}

\subsection{The main theorem}

We are finally ready to state the main results of this paper. 
Denote by $d_{TV}(F,G)$ the total variation distance between the laws of two random variables $F$ and $G$.

\begin{theorem}
Assume that \textbf{Hypotheses 2.1, 2.2. and 2.3} hold with $q^\ast\geq \frac{3}{\delta}+1$ for some $\delta  > 0$.
\smallskip 

(a) If in addition we assume \textbf{Hypothesis 2.4} (a), 
then there exists a constant $C$ depending on $\delta$ and $T$ such that
\begin{eqnarray}
	d_{TV}(X_t,X_t^\varepsilon)\leq C\eta _{3}(\varepsilon )^{1-\delta }. \label{4.000}
\end{eqnarray}

Under the above hypotheses, the laws of $X_t$ and $X_t^\varepsilon$ are absolutely continuous with respect to the Lebesgue measure, 
with smooth densities $p_{X_t}(x)$ and $p_{X_t^\varepsilon}(x)$.
Moreover, if $\ell$ is an index such that $q^{*}\geq \frac{3+\ell}{\delta}+1$, then there exists a constant ${C}$ depending on $\delta$, $T$ and $\ell$ such that 
\begin{eqnarray}
	\Vert p_{X_t}- p_{X_t^\varepsilon}\Vert_{\ell,\infty}\leq {C}\eta _{3}(\varepsilon )^{1-\delta }.  \label{4.00000}
\end{eqnarray}

(b) If in addition we assume \textbf{Hypothesis 2.4} (b),
then there exists a constant $C$ depending on $\delta$ and $T$ such that for every $t\in[0,T]$ 
with $t>\frac{8\alpha( \frac{3}{\delta}-1)}{\varepsilon_{\ast}}$ (with $\varepsilon_\ast$ and $\alpha$ given in \textbf{Hypothesis 2.4}\ (b)), we have
\begin{eqnarray}
	d_{TV}(X_t,X_t^\varepsilon)\leq C\eta _{3}(\varepsilon )^{1-\delta }. \label{4.b1}
\end{eqnarray}

For any index $\ell$ and for $t>\frac{8\alpha(3 \ell +2)}{\varepsilon_\ast}$, both the laws of $X_t$ and $X_t^\varepsilon$ have $\ell$-times differentiable densities $p_{X_t}(x)$ and $p_{X_t^\varepsilon}(x)$. Assume moreover that $q^\ast\geq \frac{3+\ell}{\delta}+1$. 
Then there exists a constant ${C}$ depending on $\delta$ ,$T$ and $\ell$ such that for 
$t>\max\{\frac{8\alpha}{\varepsilon_\ast}( \frac{3+\ell}{\delta}-1),\frac{8\alpha(3 \ell +2)}{\varepsilon_\ast}\}$, we have
\begin{eqnarray}
	\Vert p_{X_t}- p_{X_t^\varepsilon}\Vert_{\ell,\infty}\leq {C}\eta _{3}(\varepsilon )^{1-\delta }.   \label{4.b2}
\end{eqnarray}
\end{theorem}

The proof of this theorem is left to Section 4.4.
\medskip 

\begin{remark} Some recent results concerning the weak approximation of the SDE with jumps are also given in $\cite{ref39},\cite{ref40},\cite{ref41}$ for example. But they do not concern the convergence in total variation distance.
\end{remark}

\subsection{A typical example}
For $t\in[0,T]$, we consider the following SDE driven by a L\'evy process:
\begin{eqnarray}
X_t=x+\int_0^t\sigma(X_{s-})dZ_s, \label{ex1}
\end{eqnarray}
where $(Z_t)_{t\in[0,T]}$ is a L\'evy process of L\'evy triplet $(0,0,\mu)$, with
$\mu(dz)=\mathbbm{1}_{(0,1]}(z)\frac{dz}{ z^{1+\rho}}$, $0\leq \rho< 1$.

We approximate (\ref{ex1}) by
\begin{eqnarray}
X_t^{\varepsilon}=x+\int_0^t\sigma(X_{s-}^{\varepsilon})dZ_s^{\varepsilon}+b(\varepsilon)\int_0^t\sigma(X_{s}^{\varepsilon})ds+c(\varepsilon)\int_0^t\sigma(X_s^{\varepsilon})dB_s, \label{ex2}
\end{eqnarray}
where $(Z_t^{\varepsilon})_{t\in[0,T]}$ is a L\'evy process of L\'evy triplet $(0,0,\mathbbm{1}_{\{z>\varepsilon\}}\mu(dz))$, $(B_t)_{t\in[0,T]}$ is a standard Brownian motion independent of $(Z_t^{\varepsilon})_{t\in[0,T]}$, and
\[
	b(\varepsilon)=\int_{(0,\varepsilon]}z\mu(dz),\quad c(\varepsilon)=\sqrt{\int_{(0,\varepsilon]}z^2\mu(dz)}.
\]

Then we have the following theorem.
\begin{theorem}
	We assume that $\sigma\in C_b^\infty(\mathbb{R})$, $0<\underline{\sigma}\leq\sigma(x)\leq\bar{\sigma}$ and $-1<a\leq\sigma'(x)\leq\bar{\sigma},\ \forall x\in\mathbb{R}$, for some universal constants $\bar{\sigma},\underline{\sigma},a$, where $\sigma'$ is the differential of $\sigma$ in $x$. 
	Then for any $\delta>0$, there is a constant $C>0$ such that for any $t\in[0,T]$,
\[
d_{TV}(X_t,X_t^\varepsilon)\leq C\varepsilon^{3-\rho-\delta}.
\]
Moreover, the laws of $X_t$ and $X_t^\varepsilon$ have smooth densities $p_{X_t}(x)$ and $p_{X_t^\varepsilon}(x)$ respectively. And for any index $l$ and any $\delta>0$, there exists a constant $C>0$ such that 
\[
\Vert p_{X_t}- p_{X_t^\varepsilon}\Vert_{l,\infty}\leq C\varepsilon^{3-\rho-\delta}.
\]
\end{theorem}
\begin{proof}
We notice that 
\[
Z_t=\int_0^t\int_{(0,1]}zN_\mu(ds,dz),
\]
where $N_\mu$ is a Poisson point measure with intensity $\mu(dz)ds$. Then (\ref{ex1}) coincides with (\ref{1.1}) with  
$c(s,z,x)=\sigma(x)z$,
and (\ref{ex2}) coincides with (\ref{1.4}) with $c(s,z,x)=\sigma(x)z$, $b_\varepsilon(s,x)=b(\varepsilon)\sigma(x)$, and $\int_{\{z\leq\varepsilon\}}zW_\mu(ds,dz)=c(\varepsilon)dB_s$.

Let $\theta:(0,1]\rightarrow[1,\infty)$ be a function defined by $\theta(z)=\frac{1}{z}$. By a change of variables, 
\[
\widetilde{c}(s,z,x)=c(s,\frac{1}{z},x)=\sigma(x)\times\frac{1}{z},\quad \nu(dz)=\mu\circ \theta^{-1}(dz)=\mathbbm{1}_{[1,\infty)}(z)\frac{dz}{z^{1-\rho}}.
\]

One can easily check that \textbf{Hypothesis 2.1} is verified (for every $q^\ast\in\mathbb{N}$) with
$\bar{c}(z)=\bar{\sigma}\times\frac{1}{z}$ and \[\int_{1}^{\infty}\vert\bar{c}(z)\vert^p\nu(dz)=\int_{1}^{\infty}\frac{\bar{\sigma}^p}{z^{p+1-\rho}}dz<\infty, \ \forall p\geq1.\]

We recall that $I_k=[k,k+1), k\in\mathbb{N}$ and $m_{k}=\nu(I_{k})$. Then for sufficiently large $z$, we have
\[
\min\{\vert \partial_z\widetilde{c}(s,z,x)\vert^2,\vert \widetilde{c}(s,z,x)\vert^2\}\geq\underline{\sigma}^2\times\frac{1}{z^4}\geq e^{-z^{\alpha_2 }},
\]
with some $0<\alpha_2<1$. We also have  \[\mathbbm{1}_{I_k}(z)\frac{\nu(dz)}{m_{k}}\geq\mathbbm{1}_{I_k}(z)\frac{1}{2}dz\geq \mathbbm{1}_{I_k}(z)\frac{1}{2}\frac{dz}{z^{1-\alpha_0}},\] with some $\alpha_0\in(\alpha_2,1)$. Moreover, since for any $p\geq1$, $p+1-\alpha_0>1$, we have
\[
\int_1^\infty\frac{\vert\bar{c}(z)\vert^p}{z^{1-\alpha_0}}dz=\bar{\sigma}\int_1^\infty\frac{1}{z^{p+1-\alpha_0}}dz<\infty.
\]So \textbf{Hypothesis 2.3} and \textbf{Hypothesis 2.4} $(a)$ are satisfied. Finally,
\[
\big\vert\frac{\partial_x\widetilde{c}(s,z,x)}{1+\partial_x\widetilde{c}(s,z,x)}\big\vert\leq \frac{\bar{\sigma}\times\frac{1}{z}}{1+a\times\frac{1}{z}}\leq \max\{\frac{1}{1+a},1\}\times\bar{\sigma}\times\frac{1}{z},
\]
so \textbf{Hypothesis 2.2} is satisfied as well. 
Then we can apply \textbf{Theorem 2.2}$(a)$ for the equation (\ref{ex1}) and (\ref{ex2}). Since 
\[
\eta_3(\varepsilon)=\int_{(0,\varepsilon]}\bar{\sigma}^3\times z^3\mu(dz)=\frac{\bar{\sigma}^3}{3-\rho}\varepsilon^{3-\rho},
\]
we obtain the estimates from \textbf{Theorem 2.3}.
\end{proof}

\bigskip

\section{Abstract integration by parts framework} %
In order to obtain the main theorem (\textbf{Theorem 2.2}), we will apply some techniques of Malliavin calculus. So firstly, we give the abstract integration by parts framework  introduced in $\cite{ref7}$. This is a variant of the integration by parts framework given in $\cite{ref10}$.

We consider a probability space ($\Omega$,$\mathcal F$,$\mathbb{P}$), and a subset $\mathcal S\subset\mathop{\bigcap}\limits_{p=1}^\infty L^p(\Omega;\mathbb{R})$ such that for every $\phi\in C_p^\infty(\mathbb{R}^d)$ and every $F\in\mathcal S^d$, we have $\phi(F)\in\mathcal S$ (with $C_p^\infty$ the space of smooth functions which, together with all the derivatives, have polynomial growth). A typical example of $\mathcal{S}$ is the space of simple functionals, as in the standard Malliavin calculus. Another example is the space of "Malliavin smooth functionals".

Given a separable Hilbert space $\mathcal{H}$, we assume that we have a derivative operator $D: \mathcal S\rightarrow\mathop{\bigcap}\limits_{p=1}^\infty L^p(\Omega;\mathcal{H})$ which is a linear application which satisfies

$a)$ 
\begin{eqnarray}
D_hF:=\langle DF,h\rangle_\mathcal{H}\in\mathcal{S},\ for\ any\ h\in\mathcal{H}, \label{0.a}
\end{eqnarray}

$b)$ $\underline {Chain\ Rule}$: For every $\phi\in C_p^\infty(\mathbb{R}^d)$ and $F=
(F_1,\cdots,F_d)\in \mathcal S^d$, we have

\begin{align}
  D\phi(F)=\sum_{i=1}^d\partial_i\phi(F)DF_i , \label{0.00}
\end{align}

Since $D_hF\in\mathcal{S}$, we may define by iteration the derivative operator of higher order $D^q:\mathcal S\rightarrow\mathop{\bigcap}\limits_{p=1}^\infty L^p(\Omega;\mathcal{H}^{\otimes q})$ which verifies
$\langle D^qF,\otimes_{i=1}^q h_i\rangle_{\mathcal{H}^{\otimes q}}=D_{h_q}D_{h_{q-1}}\cdots D_{h_1}F$. We also denote $D_{h_1,\cdots,h_q}^qF:=D_{h_q}D_{h_{q-1}}\cdots D_{h_1}F$. Then, $D_{h_1,\cdots,h_q}^qF=$
$D_{h_q}D_{h_1,\cdots,h_{q-1}}^{q-1}F$ ($q\geq2$).

For $F=(F_1,\cdots,F_{\Tilde{d}})\in\mathcal{S}^{\Tilde{d}}$, we define $\sigma_F=(\sigma_F^{i,j})_{i,j=1,\cdots,{\Tilde{d}}}$ to be the Malliavin covariance matrix
 with $\sigma_F^{i,j}=\langle DF_i,DF_j\rangle_\mathcal{H}$ and we denote 
\begin{eqnarray}
\Sigma_p(F)=\mathbb{E}(1/ \det \sigma_F)^p.  \label{sigma}
\end{eqnarray}
For ${\Tilde{d}}=1$, which is the case that we discuss in this paper, $\det\sigma_F=\sigma_F=\langle DF,DF\rangle_\mathcal{H}$.
We say that $F$ is non-degenerated if $\Sigma_p(F)<\infty$, $\forall p\geq1$.

We also assume that we have an Ornstein-Uhlenbeck operator $L:\mathcal S\rightarrow\mathcal S$ which is a linear operator satisfying the following duality formula:\\
$\underline {Duality}$: For every $F,G\in\mathcal S$,

\begin{align}
\mathbb{E}\langle DF,DG\rangle_\mathcal{H}=\mathbb{E}(FLG)=\mathbb{E}(GLF).
\label{0.01}
\end{align}
As an immediate consequence of the duality formula, we know that $L: \mathcal{S}\subset L^2(\Omega)\rightarrow L^2(\Omega)$ is closable.

\begin{definition}
If $D^q: \mathcal{S}\subset L^2(\Omega)\rightarrow L^2(\Omega;\mathcal{H}^{\otimes q})$, $\forall q\geq1$, are closable, then the triplet $(\mathcal{S},D,L)$ will be called an IbP (Integration by Parts) framework.
\end{definition}

Now, we introduce the Sobolev norms. For any $l\geq1$, $F\in\mathcal{S}$,
\begin{eqnarray}
\left\vert F\right\vert_{1,l} &=&\sum_{q=1}^{l}\left\vert D^{q}F\right\vert_{\mathcal{H}^{\otimes q}},\quad \left\vert F\right\vert_{l}=\left\vert F\right\vert+\left\vert F\right\vert_{1,l}.\label{norm}
\end{eqnarray}

We remark that $\vert F\vert_{0}=\vert F\vert$ and $\vert F\vert_{1,l}=0$ for $l=0$. For $F=(F_1,\cdots,F_d)\in\mathcal{S}^d$, we set 
\begin{eqnarray}
\left\vert F\right\vert_{1,l} &=&\sum_{i=1}^{d}\left\vert F_i\right\vert_{1,l},\quad \left\vert F\right\vert_{l}=\sum_{i=1}^{d}\left\vert F_i\right\vert_l.\nonumber
\end{eqnarray}

Moreover, we associate the following norms. For any $l, p\geq1$,
\begin{eqnarray}
\left\Vert F \right\Vert _{l,p}&=&(\mathbb{E}\left\vert F \right\vert_{l}^{p})^{1/p}, \quad \left\Vert F \right\Vert _{p}=(\mathbb{E}\left\vert F \right\vert^{p})^{1/p} ,\nonumber \\
\left\Vert F\right\Vert _{L,l,p}&=&\left\Vert F\right\Vert _{l,p}+\left\Vert
LF\right\Vert _{l-2,p}. 
\end{eqnarray}

We denote by $\mathcal{D}_{l,p}$ the closure of $\mathcal{S}$ with respect to the norm $\left\Vert \circ
\right\Vert _{L,l,p}:$ 
\begin{equation}
\mathcal{D}_{l,p}=\overline{\mathcal{S}}^{\left\Vert \circ \right\Vert
_{L,l,p}},  \label{3.9}
\end{equation}%
and
\begin{equation*}
    \mathcal{D}_{\infty}=\mathop{\bigcap}\limits_{l=1}^{\infty}\mathop{\bigcap}\limits_{p=1}^{\infty}\mathcal{D}_{l,p},\quad \mathcal{H}_{l}=\mathcal{D}_{l,2}.
\end{equation*}

For an IbP framework $(\mathcal{S},D,L)$, we now extend the operators from $\mathcal{S}$ to $\mathcal{D}_{\infty}$.
For $F\in \mathcal{D}_{\infty}$, $p\geq2$, there exists a sequence $F_{n}\in \mathcal{S}$ such that $\left\Vert F-F_{n}\right\Vert _{p}\rightarrow 0$, $\left\Vert F_{m}-F_{n}\right\Vert _{q,p}\rightarrow 0$ and $\left\Vert LF_{m}-LF_{n}\right\Vert _{q-2,p}\rightarrow 0$. Since $%
D^{q} $ and $L$ are closable, we can define
\begin{equation}
D^{q}F=\lim_{n\rightarrow\infty}D^{q}F_{n}\quad in\quad L^p(\Omega;\mathcal{H}^{\otimes q}),\quad LF=\lim_{n\rightarrow\infty}LF_{n}\quad in\quad L^p(\Omega).  \label{3.10}
\end{equation}%
We still associate the same norms introduced above for $F\in\mathcal{D}_\infty$.

\begin{lemma}
The triplet $(\mathcal{D}_\infty,D,L)$ is an IbP framework.
\end{lemma}

\begin{proof}
Here we just show that $D$ verifies (\ref{0.a}): For $F\in\mathcal{D}_\infty$ and $h\in\mathcal{H}$, we have $\langle DF,h\rangle_\mathcal{H}\in\mathcal{D}_\infty$.

In fact, for any $k\geq1,p\geq2$, any $F\in\mathcal{D}_{k+1,p}$, there is a sequence $F_n\in\mathcal{S}$ such that $\Vert F_n-F\Vert_{k+1,p}\rightarrow0$. Then for any $u_1,\cdots,u_k\in L^p(\Omega;\mathcal{H}),h\in\mathcal{H}$, any $n,m\in\mathbb{N},$
\begin{align*}
    &\mathbb{E}\langle D^k(\langle DF_m,h\rangle_\mathcal{H}-\langle DF_n,h\rangle_\mathcal{H}),u_1\otimes\cdots\otimes u_k\rangle_{\mathcal{H}^{\otimes k}}^{\frac{p}{2}}=\mathbb{E}\vert D_{u_k}D_{u_{k-1}}\cdots D_{u_1}\langle D(F_m-F_n),h\rangle_\mathcal{H}\vert^{\frac{p}{2}}\\
    &=\mathbb{E}\vert D_{u_k}D_{u_{k-1}}\cdots D_{u_1}D_h(F_m-F_n)\vert^{\frac{p}{2}}=\mathbb{E}\vert\langle D^{k+1}(F_m-F_n),h\otimes u_1\otimes\cdots\otimes u_k\rangle_{\mathcal{H}^{\otimes (k+1)}}\vert^{\frac{p}{2}}\\
    &\leq\mathbb{E}\vert D^{k+1}(F_m-F_n)\vert_{\mathcal{H}^{\otimes (k+1)}}^{\frac{p}{2}}\vert h\otimes u_1\otimes\cdots\otimes u_k\vert_{\mathcal{H}^{\otimes (k+1)}}^{\frac{p}{2}}\rightarrow0,
\end{align*}
which yields that $\mathbb{E}\vert D^k(\langle DF_m,h\rangle_\mathcal{H}-\langle DF_n,h\rangle_\mathcal{H})\vert_{\mathcal{H}^{\otimes k}}^p\rightarrow0.$ Therefore, $\langle DF,h\rangle_{\mathcal{H}}\in\mathcal{D}_{k,p}$ and (\ref{0.a}) is verified.
\end{proof}

The following lemma is useful in order to control the Sobolev norms $\left\Vert F\right\Vert
_{L,l,q}$.
\begin{lemma}
We fix $p\geq 2,l\geq2.$ Let $F\in  L^1(\Omega)$ and let $F_{n}\in \mathcal{S},n\in\mathbb{N}$ such that 
\begin{eqnarray*}
i)\quad \mathbb{E}\left\vert F_{n}-F\right\vert &\rightarrow &0, \\
ii)\quad \sup_{n}\left\Vert F_{n}\right\Vert _{L,l,p} &\leq &K_{l,p}<\infty. 
\end{eqnarray*}%
Then for every $1\leq \bar{p}<p,$ we have $F\in \mathcal{D}_{l,\bar{p}}$ and $\left\Vert F\right\Vert
_{L,l,\bar{p}}\leq K_{l,\bar{p}}$ .
\end{lemma}

\begin{proof}
The Hilbert space $\mathcal{H}_{l}=\mathcal{D}_{l,2}$ equipped with the scalar product
\begin{eqnarray*}
\left\langle U,V\right\rangle _{L,l,2}&:=&\sum_{q=1}^{l}\mathbb{E} \langle D^qU, D^qV\rangle_{\mathcal{H}^{\otimes q}}+\mathbb{E}\vert UV\vert\\
&+&\sum_{q=1}^{l-2}\mathbb{E} \langle D^qLU, D^qLV\rangle_{\mathcal{H}^{\otimes q}}+\mathbb{E}\vert LU\times LV\vert
\end{eqnarray*}
is the space of the functionals
which are $l-$times differentiable in $L^{2}$ sense.
By $ii)$, for $p\geq2$, $\left\Vert
F_{n}\right\Vert _{L,l,2}\leq \left\Vert F_{n}\right\Vert _{L,l,p}\leq K_{l,p}$.
Then, applying Banach Alaoglu's theorem, there exists a functional $G\in\mathcal{H}_l$ and a subsequence (we still denote it by $n$), such that $F_n\rightarrow G$ weakly in the Hilbert space $\mathcal{H}_l$. This means that for every $Q\in\mathcal{H}_l$, $\langle F_n,Q\rangle_{L,l,2}\rightarrow\langle G,Q\rangle_{L,l,2}$. Therefore, by Mazur's theorem, we can construct some convex combination
\[
G_{n}=\sum_{i=n}^{m_{n}}\lambda _{i}^{n}\times F_{i}\in \mathcal{S}
\]%
with $\lambda _{i}^{n}\geq 0,i=n,....,m_{n}$ and $%
\sum\limits_{i=n}^{m_{n}}\lambda _{i}^{n}=1$,
such that 
\[
\left\Vert G_{n}-G\right\Vert _{L,l,2}\rightarrow 0.
\]
In particular we have 
\[
\mathbb{E}\left\vert G_{n}-G\right\vert \leq \left\Vert G_{n}-G\right\Vert
_{L,l,2}\rightarrow 0.
\]%
Also, we notice that by i),
\[
\mathbb{E}\left\vert G_{n}-F\right\vert \leq \sum_{i=n}^{m_{n}}\lambda
_{i}^{n}\times \mathbb{E}\left\vert F_{i}-F\right\vert \rightarrow 0.
\]%
So we conclude that 
$F=G\in \mathcal{H}_{l}.$
Thus, we have 
\[
\mathbb{E}(\left\vert G_{n}-F\right\vert
_{l}^{2})+\mathbb{E}(\left\vert LG_{n}-LF\right\vert
_{l-2}^{2})\leq \left\Vert G_{n}-F\right\Vert _{L,l,2}^{2}\rightarrow 0.
\]%
By passing to a subsequence, we have $\left\vert G_{n}-F\right\vert _{l}+\left\vert LG_{n}-LF\right\vert _{l-2}\rightarrow 0
$ almost surely. 
Now, for every $\bar{p}\in[1,p)$, we denote $Y_{n}:=\left\vert G_{n}\right\vert
_{l}^{\bar{p}}+\left\vert LG_{n}\right\vert
_{l-2}^{\bar{p}}$ and $Y:=\left\vert F\right\vert _{l}^{\bar{p}}+\left\vert LF\right\vert _{l-2}^{\bar{p}}$. Then, $Y_n\rightarrow Y$ almost surely, and for any $\Tilde{q}\in[\bar{p},p]$,
\begin{eqnarray}
\mathbb{E}\vert G_n\vert_{l}^{\Tilde{q}}+\mathbb{E}\vert LG_n\vert_{l-2}^{\Tilde{q}}&\leq&\left\Vert G_{n}\right\Vert_{L,l,\Tilde{q}}^{\Tilde{q}}=
\left\Vert \sum_{i=n}^{m_{n}}\lambda _{i}^{n}\times F_{i}\right\Vert_{L,l,\Tilde{q}}^{\Tilde{q}}
\leq (\sum_{i=n}^{m_{n}}\lambda _{i}^{n}\times \left\Vert
F_{i}\right\Vert_{L,l,\Tilde{q}})^{\Tilde{q}}  \nonumber    \\
&\leq &(\sup_{i}\left\Vert F_{i}\right\Vert_{L,l,\Tilde{q}}\times\sum_{i=n}^{m_{n}}\lambda _{i}^{n})^{\Tilde{q}} =\sup_{i}\left\Vert F_{i}\right\Vert_{L,l,\Tilde{q}}^{\Tilde{q}}\leq K_{l,\Tilde{q}}^{\Tilde{q}}. \nonumber
\end{eqnarray}
So $(Y_n)_{n\in\mathbb{N}}$ is uniformly integrable, and we have
\[
\left\Vert F\right\Vert _{L,l,\bar{p}}^{\bar{p}}=\mathbb{E}(\left\vert F\right\vert
_{l}^{\bar{p}})+\mathbb{E}(\left\vert LF\right\vert
_{l-2}^{\bar{p}})=\mathbb{E}(Y)=\lim_{n\rightarrow\infty}\mathbb{E}(Y_{n})\leq K_{l,\bar{p}}^{\bar{p}},
\]%
\end{proof}

\subsection{Main consequences: Convergence in total variation distance}
We will use the abstract framework in $\cite{ref7}$ for the IbP framework $(\mathcal{D}_\infty,D,L)$, with $D$ and $L$ defined in (\ref{3.10}). 
Using Malliavin type arguments, $\cite{ref7}$ proves the following results. 
The first result, concerning the density, is classical:
\begin{lemma}
Let $F\in \mathcal{D}_{\infty}$. If $\Sigma_{6p+4}(F)<\infty$,
then the law of random variable $F$ has a density $p_F(x)$ which is $p-$times differentiable.
\end{lemma}
In the following, we define the distances between random variables $F,G:\Omega
\rightarrow \mathbb{R}$:%
\[
d_{k}(F,G)=\sup \{\left\vert \mathbb{E}(f(F))-\mathbb{E}(f(G))\right\vert :\sum_{0\leq i\leq k}\left\Vert  f^{(i)}\right\Vert _{\infty }\leq 1\} 
\]%
For $k=1$, this is the Fortet Mourier
distance (which is a variant of the Wasserstein distance), while for $k=0$, this is
the total variation distance and we denote it by $d_{TV}$.
Now we present the second result concerning the total variation distance:

\begin{lemma}
We fix some index $l$, some $r\in \mathbb{N}$ and some $\delta >0.$ We define $p_1=2( r(\frac{1}{\delta}-1)+2)$, $p_2=\max\{6l+4,2( \frac{r+l}{\delta}-r+2)\}$, $q_1\geq r(\frac{1}{\delta}-1)+4$, $q_2\geq\frac{r+l}{\delta}-r+4$. Let $F,G\in \mathcal{D}_{\infty}.$  Then one may find $C\in \mathbb{R}_{+}$
, $p\in \mathbb{N}$ (depending on $r$ and $\delta )$ such that 
\begin{equation}
i)\quad d_{TV}(F,G)\leq C(1+\Sigma _{p_1}(F)+\Sigma_{p_1}(G)+\left\Vert F\right\Vert
_{L,q_1,p}+\left\Vert G\right\Vert _{L,q_1,p})\times
d_{r}(F,G)^{1-\delta },
\label{3.12}
\end{equation}%
and
\begin{equation}
ii)\quad\Vert p_F- p_G\Vert_{l,\infty}\leq C(1+\Sigma _{p_2}(F)+\Sigma_{p_2}(G)+\left\Vert F\right\Vert
_{L,q_2,p}+\left\Vert G\right\Vert _{L,q_2,p})\times
d_{r}(F,G)^{1-\delta },  \label{3.12*}
\end{equation}%
where $p_F(x)$ and $p_G(x)$ denote the density functions of $F$ and $G$ respectively.
\end{lemma}

\textbf{Comment }The significance of this lemma is the following. Suppose that one has already obtained an estimate of a "smooth" distance $d_{r}$ between two random vectors $F$ and $G$ (in our case $r=3$ in (\ref{1.21})). But we would like to control the total variation distance between them. In order
to do this, one employs some integration by parts techniques which are
developed in $\cite{ref7}$ and conclude the following. We need to assume that both $F$ and $G$
are "smooth" in the sense that $\left\Vert F\right\Vert _{L,q,p}+\left\Vert
G\right\Vert _{L,q,p}<\infty $ for sufficiently large $q,p.$ Moreover, we need some non
degeneracy condition: both $F$ and $G$ are non-degenerated, that is $%
\Sigma _{p}(F)+\Sigma _{p}(G)<\infty $, with $p$ large enough. Then (\ref{3.12}%
) asserts that one may control $d_{TV}$ by $d_{r},$ and the control is quasi
optimal: we loose just a power $\delta >0$ which we may  take as small
as we want. And (\ref{3.12*}) says that we may also control the distance between the derivatives of density functions by $d_r$.

Then we can get the following corollary.
\begin{corollary}
We fix some index $l$, some $r\in \mathbb{N}$ and some $\delta >0.$ We define $p_1, p_2, q_1, q_2$ as in \textbf{Lemma 3.4}.
Let $F_{M}\in \mathcal{D}_{\infty },M\in \mathbb{N}$ such that for every $p\geq1$,
\[
\sup_M(\left\Vert F_{M}\right\Vert _{L,q_1,p}+\Sigma _{p_1}(F_{M}))\leq
Q_{q_1,p,p_1}<\infty ,
\]with $Q_{q_1,p,p_1}$ a constant not dependent on $M$.
Consider moreover some random variable $F$ such that $d_{r}(F,F_{M})\rightarrow
0.$ Then there exists a constant $C>0$ such that
\[
i)\quad d_{TV}(F,F_{M})\leq Cd_{r}(F,F_{M})^{1-\delta }.
\]%
Moreover, if $\sup\limits_M(\left\Vert F_{M}\right\Vert _{L,q_2,p}+\Sigma _{p_2}(F_{M}))\leq
Q_{q_2,p,p_2}<\infty$, the law of $F$ is absolutely continuous with smooth density $p_{F}$
and one has
\[
ii)\quad \left\Vert p_{F}-p_{F_{M}}\right\Vert _{l,\infty }\leq
Cd_{r}(F,F_{M})^{1-\delta }.
\]
\end{corollary}

\begin{proof}
We take $C$ to be a constant depending on $p,p_1,q_1,r$ and $\delta$ which can change from one line to another. By \textbf{Lemma 3.4}, for every $M<M^{\prime}$, one has
\begin{eqnarray}
d_{TV}(F_{M},F_{M^{\prime}})\leq Cd_{r}(F_{M},F_{M^{\prime}})^{1-\delta }\leq C[d_{r}(F_{M},F)^{1-\delta }+d_{r}(F,F_{M^{\prime}})^{1-\delta }].\label{Coro341}
\end{eqnarray}So $(F_M)_{M\in\mathbb{N}}$ is a Cauchy sequence in $d_{TV}$. It follows that it has a limit $G$. But since $d_r(F_M,F)\rightarrow0$, it follows that $F=G$. Passing to the limit $M^{\prime}\rightarrow\infty$ in (\ref{Coro341}), we get
\[
d_{TV}(F_{M},F)\leq Cd_{r}(F_{M},F)^{1-\delta }.
\]
The proof of $ii)$ is analogous.
\end{proof}

\bigskip

\bigskip

\section{Malliavin calculus and stochastic differential equations with jumps}
In this section we present the integration by parts framework that will be
used in the following. To begin we give a quick informal presentation of our
strategy. We will work with the solution of the equation (\ref{1.4}), but, for
technical reasons, we make the change of variable $z\mapsto \frac{1}{z}$
so the equation of interest is now the equation (\ref{3.02}). We use the notation
from that section. The intensity measure for our random measure is $%
\mathbbm{1}_{[1,M)}(z)\nu (dz)ds$ and this is a finite measure. Then the corresponding
Poisson Point measure $N_{\nu }$ may be represented by means of a 
compound
Poisson process. For some technical reasons, we produce the representation on each set $\{z\in I_k=[k,k+1)\}, k\in\mathbb N$, so the equation (\ref{3.02}) reads%
\begin{eqnarray*}
\widehat{X}_{t}^{M } &=&x+\int_{0}^{t}\sum_{k=1}^{M-1}\int_{\{z \in
I_k\}}\widetilde{c}(s,z,\widehat{X}_{s-}^{M})N_{\nu
}(ds,dz)   \nonumber\\
&&+\int_{0}^{t}b_{M}(s,\widehat{X}_s^{M})ds+\int_{0}^{t}\int_{\{z\geq M\}}\widetilde{c}(s,z,\widehat{X}_{s}^{M})W_\nu(ds,dz)   \nonumber\\
&=&x+\sum_{k=1}^{M-1}\sum_{i=1}^{J_t^k}\widetilde{c}(T_i^k,Z_i^k,\widehat{X}_{T_i^k-}^{M})   \nonumber\\
&&+\int_{0}^{t}b_{M}(s,\widehat{X}_s^{M})ds+\int_{0}^{t}\int_{\{z\geq M\}}\widetilde{c}(s,z,\widehat{X}_{s}^{M})W_\nu(ds,dz)  .
\end{eqnarray*}%
Here $T_i^k, k,i\in\mathbb{N}$ are the jump times of the Poisson process $(J_t^k)_{t\in[0,T]}$ of parameter $\nu(I_k)$, and  $Z_i^k, k,i\in \mathbb{N}$ are independent random
variables of law $\mathbbm{1}_{I_k}(z)\frac{\nu(dz)}{\nu(I_k)}$, which are independent of $J^k$ as well. We will work
conditionally to $T_i^k, k,i\in\mathbb{N}$, so the randomness in the system comes from $%
W_{\nu }$ on one hand and from $Z_i^k, k,i\in \mathbb{N}$ on the other hand. Concerning $%
W_{\nu }$ we will use the standard Malliavin calculus (which fits in the
framework presented in Section 3). But we will also use this integration by parts calculus
with respect to the amplitude of the jumps given by $Z_i^k, k,i\in \mathbb{N}.$ We
present this kind of calculus now.

Suppose for a moment (just for simplicity) the law of $Z_i^k$ is absolutely
continuous with respect to the Lebesgue measure and has a smooth density $%
h_k(z)$ which has compact support. We also assume that the logarithm of the density $\ln h_k$ is smooth. Then we look to $\widehat{X}_{t}^{M}$ as to a functional $%
F(Z_{1}^1,...,Z^{M-1}_{J_{t}^{k}})$ and we define the derivative operators%
\begin{equation*}
D_{k,i}^{Z}F=\frac{\partial }{\partial z_{i}^k}F(Z_{1}^1,...,Z^{M-1}_{J_{t}^{k}}),
\end{equation*}%
and 
\begin{equation*}
L^{Z}F=-\sum_{k,i}D_{k,i}^{Z}D_{k,i}^{Z}F+D_{k,i}^{Z}F\times \partial _{z}\ln h_k(Z^k_i).
\end{equation*}%
And we check that these operators verify the conditions in Section 3. Since
we want to use integration by parts with respect to both $W_{\nu }$ and $%
Z_i^k, k,i\in \mathbb{N}$, we will consider the derivative operator $D=(D^{W},D^{Z})$ and
the operator $L=(L^{W},L^{Z})$ where $D^{W}$ and $L^{W}$ are the derivative
and Ornstein Uhlenbeck operators from the standard Malliavin calculus for Gaussian random variables. With
these operators at hand we check the hypotheses of \textbf{Lemma 3.4} and of
\textbf{Corollary 3.4.1}, and these are the results which allow as to prove our 
\textbf{Theorem 2.2}.

Roughly speaking this is our strategy. But there is one more point: the
hypotheses we raise for the law of $Z^{k}_i$ that it has a smooth density with compact support and has a smooth logarithm density, is rather strong and we want to weaken it. This is the aim of the
"splitting method". This amounts to produce three independent random
variables $V_i^{k},U_i^{k}$ and $\xi _i^{k}$ such that $Z^{k}_i$ has the same law as 
$\xi _i^{k}V_i^{k}+(1-\xi _i^{k})U_i^{k}$ with $\xi _i^{k}$ a Bernoulli random
variable and $V_i^{k}$ a random variable with good properties. So we split $%
Z^{k}_i$ in two parts, $V_i^{k}$ and $U_i^{k}.$ We may do it in such a way that $%
V_i^{k}$ has the law $\psi _{k}(v)dv$ with $\psi _{k}\in C_{c}^{\infty }(\mathbb{R})$
(see Section 4.1 for the precise procedure). And we perform the Malliavin
calculus with respect to $V_i^{k}$ instead of $Z^{k}_i$ (we work conditionally to 
$\xi _i^{k}$ and $U_i^{k}$ which appear as constants).

\subsection{The splitting method}
We consider a Poisson point measure $N_\nu(ds,dz)$ with compensator $\widehat{N}_\nu(ds,dz)=\nu (dz)ds$ on the state space $[1,\infty)$. We will make use of the noise $%
z\in [1,\infty)$ in order to apply the
results from the previous section.
We recall that $I_{k}=[k,k+1)$ and $m_{k}=\nu(I_{k})$, and we
suppose that for every $k$, there exists $\varepsilon _{k}>0$, such
that 
\begin{equation}
\mathbbm{1}_{I_k}(z)\frac{\nu (dz)}{m_k}\geq \mathbbm{1}_{I_k}(z)\varepsilon _{k} \times dz. \label{3.13}
\end{equation}%
\begin{remark}
Under \textbf{Hypothesis 2.4} $(a)$, the splitting condition (\ref{3.13}) is satisfied with 
$\varepsilon_k=\frac{\varepsilon_\ast}{(k+1)^{1-{\alpha}}}$, for any ${\alpha}\leq\alpha_1$. 
If instead we assume \textbf{Hypothesis 2.4} $(b)$, (\ref{3.13}) is also satisfied, with $\varepsilon_k=\frac{\varepsilon_\ast}{k+1}$.
\end{remark}

When (\ref{3.13}) is satisfied, we are able to use the "\textbf{splitting method" }as follows. 
To begin we define the functions
\begin{eqnarray}
	a(y) &=&1-\frac{1}{1-(4y-1)^{2}}\quad for\quad y\in \lbrack \tfrac{1}{4}, \tfrac{1}{2})  \label{3.14} \\
\psi (y) &=&\mathbbm{1}_{\{\left\vert y\right\vert \leq \frac{1}{4}\}}+\mathbbm{1}_{\{\frac{1}{4}%
< \left\vert y\right\vert \leq \frac{1}{2}\}}e^{a(\left\vert y\right\vert
)}.  \label{3.14*}
\end{eqnarray}%
We notice that $\psi \in C_{c}^{\infty }(\mathbb{R})$ and that its support is included in $[-\frac{1}{2},\frac{1}{2}]$. 
We also notice that for every $q,p\in \mathbb{N}$ the
function $y\mapsto \vert a^{(q)}(y)\vert ^{p}\psi (y)$ is
continuous and has support included in $[-\frac{1}{2},\frac{1}{2}],$
so it is bounded: one may find $C_{q,p}$ such that 
\begin{equation}
\big\vert a^{(q)}(y)\big\vert ^{p}\psi (y)\leq C_{q,p}\quad \forall y\in
\mathbb{R}.  \label{3.15}
\end{equation}%
We denote
\begin{equation}
\psi_{k}(y)=\psi(y-(k+\tfrac{1}{2})),\quad\theta _{k}(y):=\partial_{y}\ln \psi _{k}(y).  \label{3.19*}
\end{equation}%
By (\ref{3.15}) (which is uniform with respect to $y$), we have
\begin{equation}
	\sup_{k}\big\vert (\ln \psi _{k})^{(q)}(y)\big\vert ^{p}\psi _{k}(y)
	\leq C_{q,p}\quad \forall y\in \mathbb{R}.  \label{3.19}
\end{equation}%
We denote%
\begin{equation}
m(\psi )=\int_{-1/2}^{1/2}\psi (y)dy.  \label{3.16}
\end{equation}

We consider a sequence of independent random variables $Z^{k}$ such that 
\[
Z^{k}\sim \mathbbm{1}_{I_{k}}(z)\frac{1}{m_{k}}\nu (dz). 
\]%
This is the sequence of random variables which are involved in the
representation of the measure $N_\nu(ds,dz)$ as long as $z\in [1,\infty)$ is
concerned. We notice that, according to our hypothesis (\ref{3.13}),
\[
\mathbb{P}(Z^{k}\in
dz)=\mathbbm{1}_{I_k}(z)\frac{\nu (dz)}{m_k}\geq \mathbbm{1}_{I_k}(z)\varepsilon _{k} \times dz. 
\]%
Then we construct some independent random variables $V^{k},U^{k},\xi ^{k}$
with laws%
\begin{eqnarray}
\mathbb{P}(V^{k} \in dz)&=&\frac{1}{m(\psi )}\psi ({z-(k+\frac{1}{2})})dz \nonumber\\
\mathbb{P}(U^{k} \in dz)&=&\frac{1}{1-\varepsilon _{k}m(\psi )}(\mathbb{P}(Z^{k}\in
dz)-\varepsilon _{k}\psi ({z-(k+\frac{1}{2})})dz) \label{split}  \\ 
\mathbb{P}(\xi ^{k} =1)&=&\varepsilon _{k}m(\psi ),\quad \mathbb{P}(\xi
^{k}=0)=1-\varepsilon _{k}m(\psi ). \nonumber
\end{eqnarray}%
We choose $\varepsilon _{k}<1/m(\psi )$ so that $%
1-\varepsilon _{k}m(\psi )>0.$ Using (\ref{3.13}), one may check that $\mathbb{P}(U^{k}\in dz)$ is a positive measure and has mass one. So it is a probability measure. And
finally one can easily check the identity of laws:
\begin{equation}
Z^{k}\sim \xi ^{k}V^{k}+(1-\xi ^{k})U^{k}.  \label{3.18}
\end{equation}

In the following, we will work directly with $Z^{k}= \xi
^{k}V^{k}+(1-\xi ^{k})U^{k}.$ This is possible because all the results that
we discuss here concern the law of the random variables, and the law remains
unchanged.

\bigskip

The Poisson point measure $N_\nu$ can be written as the following sum:
\[
N_\nu(ds,dz)=\sum_{k=1}^\infty\mathbbm{1}_{I_k}(z)N_\nu(ds,dz)=\sum_{k=1}^\infty N_{\nu_k}(ds,dz),
\]where $\nu_k(dz)=\mathbbm{1}_{I_k}(z)\nu(dz)$ and $N_{\nu_k}$ is a Poisson point measure with intensity $\nu_k(dz)ds$.

The Poisson point measure $N_{\nu_k}$ can be represented by means of compound Poisson processes as follows. For each $k\in\mathbb{N}$, we denote by $T_i^k,i\in\mathbb{N}$ the jump times of a Poisson process $(J_t^k)_{t\in[0,T]}$ of parameter $m_k$, and we consider a sequence of independent random variables $Z_i^k\sim\mathbbm{1}_{I_k}(z)\frac{\nu(dz)}{m_k}, i\in\mathbb{N}$, which are independent of $J^k$ as well. Then, for any $t>0$ and $A\in\mathcal{B}([k,k+1))$,
$N_{\nu_k}([0,t]\times A)=\sum\limits_{i=1}^{J_t^k}\mathbbm{1}_{A}(Z_i^k).$
And for each $k,i\in\mathbb{N}$, we will split $Z_i^k$ as 
$Z_i^k=\xi_i^kV_i^k+(1-\xi_i^k)U_i^k.$

\begin{remark} 
	The law of $Z^k_i$ could be very irregular and it is not possible to make integration by parts based on it. So we construct the $V^k_i$, which has all the good regularity properties in order to make Malliavin calculus. This is the idea of the splitting method. The splitting method presented here is analogous to the one in $\cite{ref10}$. Therein, Bichteler, Gravereau and Jacod deal with 2 kinds of independent Poisson point measures. One is very regular, and  smooth enough to make Malliavin calculus on it (in our paper, $V^k_i$ play the same role). The other one can be arbitrary, and it may be very irregular (in our paper, it corresponds to $U^k_i$). But the difference is that instead of splitting the Poisson point measure, we split the random variables, and so this method can also be applied in a large class of different problems. For example, Bally, Caramellino and Poly use the splitting method to show the convergence in total variation distance in the central limit theorem in $\cite{ref3}$. Other possible approaches to the Malliavin calculus for jump processes are given in the papers $\cite{ref33},\cite{ref34},\cite{ref36},\cite{ref35}$ and the book $\cite{ref37}$ for example.
\end{remark}

\subsection{Malliavin calculus for Poisson point measures and space-time Brownian motions}
In this section we present the IbP framework on a space where we have
the Poisson point measure $N_{\nu }$ presented in the previous section and
moreover we have a space-time Brownian motion $W_{\nu }(ds,dz)$ with covariance measure $\nu (dz)ds$, which is independent of $N_\nu$. We recall that in Section 2.2 we have introduced the random variables $W_\nu(\varphi), N_\nu(\varphi)$ and the filtrations $(\mathcal{F}_t^W)_{t\in[0,T]}$, $(\mathcal{F}_t^N)_{t\in[0,T]}$, and we denote $\mathcal{F}_t=\mathcal{F}_t^W\bigvee\mathcal{F}_t^N$. We present now the Malliavin calculus. 
We recall the random variables $T^k_i$, and $Z_i^k=\xi_i^kV_i^k+(1-\xi_i^k)U_i^k$ introduced in the
previous section and we take $\mathcal{G}=\sigma(U^{k}_i,\xi ^{k}_i,T_i^k: k,i\in \mathbb{N})$\ to be the $\sigma-$algebra associated to the noise
from $U^{k}_i,\xi ^{k}_i,T_i^k,\ k,i\in \mathbb{N}.$ These are the noises which will not be involved in the Malliavin calculus.
We denote by $C_{%
\mathcal{G},p}$ the space of the functions $f:\Omega \times \mathbb{R}^{m^\prime\times m}\times\mathbb{R}^{n}\rightarrow\mathbb{R}$
such that $f$ is $\mathcal{F}_T-$measurable, and for each $\omega ,$ the function $(v_{1}^1,...,v^{m}_{m^\prime},w_1,\cdots,w_{n})\mapsto f(\omega ,v_{1}^1,...,v^{m}_{m^\prime},w_1,\cdots,w_{n})$ belongs to $%
C_{p}^{\infty }(\mathbb{R}^{m^\prime\times m}\times\mathbb{R}^{n})$, and for each $(v_{1}^1,...,v^{m}_{m^\prime},w_1,\cdots,w_{n})$, the
function $\omega \mapsto f(\omega ,v_{1}^1,...,v^{m}_{m^\prime},w_1,\cdots,w_{n})$ is $%
\mathcal{G}$-measurable. Then we define the space of simple functionals
\begin{equation*}
\mathcal{S}=\{F=f (\omega ,(V_{i}^k)_{\substack{1\leq i\leq m^\prime\\1\leq k\leq m}},(W_\nu(\varphi_{j}))_{j=1}^{n}) :f\in C_{%
\mathcal{G},p}, \varphi_1,\cdots,\varphi_n\in L^2([1,\infty)\times[0,T],\nu\times Leb), m^\prime,m,n\in \mathbb{N}\}.
\end{equation*}%
On the space $\mathcal{S}$ we define the derivative operators
\begin{eqnarray*}
&& D^{Z}_{(k_0,i_0)}F=\mathbbm{1}_{\{k_0\leq m\}}\mathbbm{1}_{\{i_0\leq m^\prime\}}\xi _{i_0}^{k_0}\frac{%
\partial f}{\partial v^{k_0}_{i_0}}(\omega ,(V_{i}^k)_{\substack{1\leq i\leq m^\prime\\1\leq k\leq m}},(W_\nu(\varphi_{j}))_{j=1}^{n}),\quad k_0,i_0\in\mathbb{N}\\
&&D^{W}_{(s,z)}F=\sum_{r=1}^{n}\frac{%
\partial f}{\partial w_r}(\omega ,(V_{i}^k)_{\substack{1\leq i\leq m^\prime\\1\leq k\leq m}},(W_\nu(\varphi_{j}))_{j=1}^{n})\varphi_r(s,z),\quad (s,z)\in[0,T]\times[1,\infty).  
\end{eqnarray*}%
We regard $D^{Z}F$ as an element of the Hilbert space $l_{2}$\ (the space of
the sequences $h=(h^k_{i})_{k,i\in \mathbb{N}}$ with $\left\vert h\right\vert
_{l_{2}}^{2}=\sum_{k=1}^{\infty }\sum_{i=1}^{\infty }\vert h^k_{i}\vert^{2}<\infty$) and $D^{W}F$ as an element
of the Hilbert space $L^2([1,\infty)\times[0,T],\nu\times Leb).$ Then
\[
DF:=(D^{Z}F,D^{W}F)\in l_{2}\otimes L^2([1,\infty)\times[0,T],\nu\times Leb).
\]%

We also denote $D^{Z,W}F=DF$ and $\mathcal{H}=l_{2}\otimes L^2([1,\infty)\times[0,T],\nu\times Leb)$. And we have%
\[
\left\langle DF,DG\right\rangle_{\mathcal{H}} =\sum_{k=1}^{\infty }\sum_{i=1}^{\infty }D^{Z}_{(k,i)}F\times D^{Z}_{(k,i)}G+\int_{[0,T]\times[1,\infty)}D^{W}_{(s,z)}F\times D^{W}_{(s,z)}G\ \nu(dz)ds. 
\]%
Moreover, we define the derivatives of order $q\in \mathbb{N}$ recursively: 
\begin{equation*}
D^{Z,W,q}_{(k_{1},i_1)\cdots(k_{q},i_q),(s_{1},z_1)\cdots(s_{q},z_q)}F:=D^{Z,W}_{(k_{q},i_q),(s_q,z_q)}D^{Z,W}_{(k_{q-1},i_{q-1}),(s_{q-1},z_{q-1})}\cdots D^{Z,W}_{(k_{1},i_1),(s_1,z_1)}F,
\end{equation*}%
and we denote $D^qF=D^{Z,W,q}F$. We also denote $D^{Z,q}$ (respectively $D^{W,q}$) as the derivative $D^Z$ (respectively $D^W$) of order $q$.

We recall the function $\theta_k$ defined in (\ref{3.19*}) and we define the Ornstein-Uhlenbeck operators $L^{Z}$, $L^{W}$ and $L=L^{Z}+L^{W}$ (which
verify the duality relation), with
\begin{eqnarray*}
L^ZF&=&-\sum_{k=1}^{m }\sum_{i=1}^{m^\prime }(D^{Z}_{(k,i)}D^{Z}_{(k,i)}F+\xi^k_i D^{Z}_{(k,i)}F\times \theta _{k}(V^{k}_i)),\\
L^WF&=&\sum_{r=1}^{n}\frac{%
\partial f}{\partial w_r}(\omega ,(V_{i}^k)_{\substack{1\leq i\leq m^\prime\\1\leq k\leq m}},(W_\nu(\varphi_{j}))_{j=1}^{n})W_\nu(\varphi_r)\\
&-&\sum_{l,r=1}^{n}\frac{%
\partial^2 f}{\partial w_l\partial w_r}(\omega ,(V_{i}^k)_{\substack{1\leq i\leq m^\prime\\1\leq k\leq m}},(W_\nu(\varphi_{j}))_{j=1}^{n})\langle\varphi_l,\varphi_r\rangle_{L^2([1,\infty)\times[0,T],\nu\times Leb)}.
\end{eqnarray*}

One can check that the triplet $(\mathcal{S},D,L)$ is consistent with the IbP framework given in Section 3. The proof is left to Appendix 5.3.

In the following, we will close the operator $D^q$ and $L$, so we will use the IbP framework $(\mathcal{D}_\infty,D,L)$ associated to $(\mathcal{S},D,L)$ in \textbf{Lemma 3.1}.

\subsection{Malliavin calculus applied to stochastic differential equations with jumps}
Now we will use the IbP framework presented in Section 4.2 for the equation (\ref{1.4}). 

Let $\theta:(0,1]\rightarrow[1,\infty)$ be a function such that $\theta(z)=\frac{1}{z}$. By a change of variables, instead of dealing with equation (\ref{1.4}), it is equivalent to consider the following equation.
\begin{eqnarray}
\widehat{X}_{t}^{M } &=&x+\int_{0}^{t}\int_{[1,M)}\widetilde{c}(s,z,\widehat{X}_{s-}^{M})N_{\nu
}(ds,dz)   \nonumber\\
&&+\int_{0}^{t}b_{M}(s,\widehat{X}_{s}^{M})ds+\int_{0}^{t}\int_{\{z \geq M\}}\widetilde{c}(s,z,\widehat{X}_{s}^{M})W_\nu(ds,dz) , \label{3.02}
\end{eqnarray}%
where $M=\frac{1}{\varepsilon}$, $\nu(dz)=\mu\circ\theta^{-1}(dz)$, $\widetilde{c}(s,z,x)=c(s,\frac{1}{z},x)$,
\begin{eqnarray}
b_{M}(s,x) =\int_{\{z\geq M\}}\widetilde{c}(s,z,x)\nu (dz),  \label{3.33} 
\end{eqnarray}%
and $W_\nu$ is the space-time Brownian motion with covariance measure $\nu(dz)ds$. One can check that $\widehat{X}_{t}^{M }$ has the same law as $X_{t}^{\varepsilon }$.

Here we give two lemmas, concerning the Malliavin-Sobolev norms and the Malliavin covariance. We recall that $\varepsilon_\ast$ and $\alpha$ are introduced in \textbf{Hypothesis 2.4} (b), and $q^\ast$ is introduced in \textbf{Hypothesis 2.1}. 

\begin{lemma}
Assuming \textbf{Hypothesis 2.1} with $q^\ast\geq2$ and \textbf{Hypothesis 2.4} (either 2.4(a) or 2.4(b)), we have $\widehat{X}_t^M\in\mathcal{D}_\infty$, and for all $p\geq1, 2\leq l\leq q^\ast$, there exists a constant $C_{l,p}(T)$ depending on $l,p,x$ and $T$, such that $\sup\limits_{M}\Vert \widehat{X}_t^{M}\Vert_{L,l,p}\leq C_{l,p}(T)$.
\end{lemma}

\begin{lemma}
Assume that \textbf{Hypothesis 2.1} with $q^\ast\geq1$ and \textbf{Hypothesis 2.2, 2.3} hold true. 

a) If we also assume \textbf{Hypothesis 2.4} (a), then for every $p\geq1$, $t\in[0,T]$, we have
\begin{eqnarray}
\sup\limits_{M}\mathbb{E}(1/\sigma_{\widehat{X}_t^M})^p\leq C_p, \label{cM8}
\end{eqnarray}
with $C_p$ a constant only depending on $p$ and $T$.

b) If we assume \textbf{Hypothesis 2.4} (b), then for every $p\geq1$, $t\in[0,T]$ such that $t>\frac{4p\alpha}{\varepsilon_{\ast}}$, we have
$\sup\limits_{M}\mathbb{E}(1/\sigma_{\widehat{X}_t^M})^p\leq C_p$.
\end{lemma}

The proofs of these lemmas are rather technical and are postponed for the Appendix (Section 5.1 and 5.2).

\bigskip 

\subsection{Proof of the main result (Theorem 2.2)}

\begin{proof}
$(a)$ By \textbf{Lemma 4.1} and \textbf{Lemma 4.2}\ $a)$, we know that for any $\delta>0$, for any $p,p_1\geq1, 2\leq q\leq q^\ast$, with $q^\ast\geq \frac{3}{\delta}+1$, there exists a constant $C_{q,p,p_1}(T)$ such that for any $M\geq1$, $t\in[0,T]$, we have
\[
\Sigma_{p_1}(\widehat{X}_t^M)+\Vert \widehat{X}_t^M\Vert_{L,q,p}\leq C_{q,p,p_1}(T).
\]
By \textbf{Lemma 2.1}, we know that
\[
d_3(\widehat{X}_t,\widehat{X}_t^M)=d_3(X_t,X_t^\varepsilon)\leq C\eta_3(\varepsilon).
\]
Then applying \textbf{Corollary 3.4.1 }$i)$ for $r=3$, we have
\[
d_{TV}(X_t,X_t^\varepsilon)=d_{TV}(\widehat{X}_t,\widehat{X}_t^M)\leq C d_3(\widehat{X}_t,\widehat{X}_t^M)^{1-\delta }\leq
C_\delta{\eta _{3}}(\varepsilon )^{1-\delta }.
\]
So we obtain (\ref{4.000}). The proof of (\ref{4.00000}) is obtained by \textbf{Corollary 3.4.1 }$ii)$, since $q^\ast\geq \frac{3+l}{\delta}+1$.\\
$(b)$ The proof is almost the same. If
$t>\frac{8\alpha( \frac{3}{\delta}-1)}{\varepsilon_{\ast}},$ then by \textbf{Lemma 4.2}\ $b)$, $\Sigma_{p_1}(\widehat{X}_t^M)<\infty$ for ${p_1}= 2( \frac{3}{\delta}-1)$. So \textbf{Corollary 3.4.1 }$i)$ still holds, and we can obtain (\ref{4.b1}). For 
\[
t>\max\{\frac{8\alpha}{\varepsilon_\ast}( \frac{3+l}{\delta}-1),\frac{8\alpha(3l+2)}{\varepsilon_\ast}\},
\] by \textbf{Lemma 4.2}\ $b)$, $\Sigma_{p_2}(\widehat{X}_t^M)<\infty$ for $ {p_2}= \max\{2( \frac{3+l}{\delta}-1),6l+4\}$. So \textbf{Corollary 3.4.1 }$ii)$ still holds, and we obtain (\ref{4.b2}).
\end{proof}

\bigskip

\section{Appendix}
\subsection{Proof of Lemma 4.1}

In the following, we will only work with the measure $\nu$ supported on $[1,\infty)$ and with the processes $(\widehat{X}_t)_{t\in[0,T]}$ and $(\widehat{X}_t^M)_{t\in[0,T]}$. So in order to simplify the notation, from now on we denote $\widehat{X}_t=X_t$ and $\widehat{X}_t^M=X_t^M$. We remark that $M=\frac{1}{\varepsilon}$ is generally not an integer, but for simplicity, we assume in the following that $M$ is an integer.

Here is the idea of the proof. Since $X_t^M$ is not a simple functional, we construct first the Euler scheme $(X_{t}^{n,M})_{t\in[0,T]}$ in subsection 5.1.1 and check that $X_t^{n,M}\rightarrow X_t^M$ in $L^1$ when $n\rightarrow\infty$. We will prove that $\mathbb{E}\vert X_t^{n,M}\vert_l^p$ and $\mathbb{E}\vert LX_t^{n,M}\vert_l^p$ are bounded (uniformly in $n,M$) in subsection 5.1.3. Then based on \textbf{Lemma 3.2}, we obtain that $X^M_t\in\mathcal{D}_\infty$ and the norms $\Vert X_t^{M}\Vert_{L,l,p}$ are bounded (uniformly in $M$).

\subsubsection{Construction of the Euler scheme}
We take a time-partition $\mathcal{P}_{t}^{n}=\{r_j=\frac{jt}{n}, j=0,\cdots,n\}$ and a space-partition $\Tilde{\mathcal{P}}_{M}^{n}=\{z_j=M+\frac{j}{n}, j=0,1,\cdots\}$. We denote $\tau_{n}(r)=r_j$ when $r\in[r_j,r_{j+1})$, and denote $\gamma_{n}(z)=z_j$ when $z\in[z_j,z_{j+1})$. Let
\begin{eqnarray}
X_{t}^{n,M} &=&x+\int_{0}^{t}\int_{[1,M)}\widetilde{c}(\tau_{n}(r),z,X_{\tau_{n}(r)-}^{n,M })N_{\nu
}(dr,dz)   \nonumber\\
&&+\int_{0}^{t}b_{M }(\tau_{n}(r),X_{\tau_{n}(r)}^{n,M
})dr+\int_{0}^{t}\int_{\{z \geq
M \}}\widetilde{c}(\tau_{n}(r),\gamma_{n}(z),X_{\tau_{n}(r)}^{n,M })W_\nu(dr,dz).  \label{4.1}
\end{eqnarray}%

Then we can obtain the following lemma.
\begin{lemma}
Assume that the \textbf{Hypothesis 2.1} holds true with $q^\ast\geq1$. Then for any $p\geq1,M\geq1$, we have $\mathbb{E}\vert X_{t}^{n,M}-X_{t}^{M}\vert^p\rightarrow0$ as $n\rightarrow\infty$.
\end{lemma}
\begin{proof}
We first notice that since $\bar{c}(z)$ (in \textbf{Hypothesis 2.1}) is decreasing, $\sup\limits_{n\in\mathbb{N}}\bar{c}(\gamma_{n}(z))\leq \bar{c}(\gamma_{1}(z)).$ So
\begin{equation}\label{euler_hyp}
	\int_{1}^\infty\sup\limits_{n\in\mathbb{N}}\vert \bar{c}(\gamma_{n}(z))\vert^2\nu(dz)\leq\int_{1}^\infty\vert \bar{c}(\gamma_{1}(z))\vert^2\nu(dz)\leq \vert\bar{c}(1)\vert^2\nu[1,2]+\int_{1}^\infty\vert \bar{c}(z)\vert^2\nu(dz)\leq C<\infty.
\end{equation}

Then by the Lebesgue dominated convergence theorem, \eqref{euler_hyp} implies that 
\begin{eqnarray}
	\lim\limits_{n\rightarrow\infty}\sup\limits_{x\in\mathbb{R}}\int_0^T\int_{[1,\infty)}\vert \widetilde{c}(s,z,x)-\widetilde{c}(\tau_{n}(s),\gamma_{n}(z),x)\vert^2\nu(dz)ds=0, 	\label{complement}
\end{eqnarray}
and
\begin{eqnarray}
	\sup\limits_{n\in\mathbb{N}}\sup\limits_{x\in\mathbb{R}}\int_0^T\int_{[1,\infty)}\vert \widetilde{c}(\tau_{n}(s),\gamma_{n}(z),x)\vert^2\nu(dz)ds\leq C. \label{complement*}
\end{eqnarray}

In the following proof, $C_p(T)$ will be a constant depending on $p$ and $T$ which may be changed from line to line.
For $p\geq2$, we write
$\mathbb{E}\vert X_{t}^{n,M}-X_{t}^{M}\vert^p\leq C_p(T)(E_1+E_2+E_3)$, where 
\begin{eqnarray*}
E_1=\mathbb{E}\vert\int_0^t\int_{[1,M)}\widetilde{c}(\tau_n(r),z,X_{\tau_{n}(r)-}^{n,M })-\widetilde{c}(r,z,X_{r-}^{M })N_\nu(dr,dz)\vert^p,
\end{eqnarray*}
\begin{eqnarray*}
E_2=\mathbb{E}\vert\int_{0}^{t} b_{M }(\tau_{n}(r),X_{\tau_{n}(r)}^{n,M
})-b_{M}(r,X_{r}^{M
})dr\vert^p,
\end{eqnarray*}
\begin{eqnarray*}
E_3=\mathbb{E}\vert\int_{0}^{t}\int_{\{z \geq
M \}}\widetilde{c}(\tau_{n}(r),\gamma_{n}(z),X_{\tau_{n}(r)}^{n,M})-\widetilde{c}(r,z,X_{r}^{M })W_\nu(dr,dz)\vert^p.
\end{eqnarray*}
Then, compensating $N_\nu$ and using Burkholder's inequality (see for example the Theorem 2.11 in $\cite{ref20}$),
\begin{eqnarray*}
E_1&\leq& C_p(T)[\mathbb{E}\int_0^t(\int_{[1,M)}\vert \widetilde{c}(\tau_n(r),z,X_{\tau_{n}(r)-}^{n,M })-\widetilde{c}(r,z,X_{r-}^{M })\vert^2\nu(dz))^{\frac{p}{2}}dr\\
&+&\mathbb{E}\int_0^t\int_{[1,M)}\vert \widetilde{c}(\tau_n(r),z,X_{\tau_{n}(r)-}^{n,M })-\widetilde{c}(r,z,X_{r-}^{M })\vert^p\nu(dz)dr\\
&+&\mathbb{E}\int_0^t\vert\int_{[1,M)} \widetilde{c}(\tau_n(r),z,X_{\tau_{n}(r)-}^{n,M })-\widetilde{c}(r,z,X_{r-}^{M })\nu(dz)\vert^pdr]\\
&\leq& C_p(T)[R_n^1+((\bar{c}_2)^{\frac{p}{2}}+\bar{c}_p+(\bar{c}_1)^p)\int_0^t\mathbb{E}\vert X_{\tau_{n}(r)}^{n,M }-X_{r}^{M }\vert^pdr],
\end{eqnarray*}
with
\begin{eqnarray*}
R_n^1&=&\mathbb{E}\int_0^t(\int_{[1,M)}\vert \widetilde{c}(\tau_n(r),z,X_{r-}^{M })-\widetilde{c}(r,z,X_{r-}^{M })\vert^2\nu(dz))^{\frac{p}{2}}dr\\
&+&\mathbb{E}\int_0^t\int_{[1,M)}\vert \widetilde{c}(\tau_n(r),z,X_{r-}^{M })-\widetilde{c}(r,z,X_{r-}^{M })\vert^p\nu(dz)dr\\
&+&\mathbb{E}\int_0^t\vert\int_{[1,M)} \widetilde{c}(\tau_n(r),z,X_{r-}^{M })-\widetilde{c}(r,z,X_{r-}^{M })\nu(dz)\vert^pdr.
\end{eqnarray*}
Since $\vert \widetilde{c}(\tau_n(r),z,X_{r-}^{M })-\widetilde{c}(r,z,X_{r-}^{M })\vert^p\leq\vert2\bar{c}(z)\vert^p\in L^1(\Omega\times[1,\infty)\times[0,T],\mathbb{P}\times\nu\times Leb)$, we apply the Lebesgue's dominated convergence theorem and we obtain that $R_n^1\rightarrow0$. Next,
\begin{eqnarray*}
E_2&\leq& {C}_p(T)\mathbb{E}\int_0^t\vert\int_{\{z\geq M \}}\widetilde{c}(\tau_n(r),z,X_{\tau_{n}(r)}^{n,M })-\widetilde{c}(r,z,X_{r}^{M })\nu(dz)\vert^pdr\\
&\leq& {C}_p(T)[R_n^2+(\bar{c}_1)^p\int_0^t\mathbb{E}\vert X_{\tau_{n}(r)}^{n,M }-X_{r}^{M }\vert^pdr],
\end{eqnarray*}
with 
\[
R_n^2=\mathbb{E}\int_0^t\vert\int_{\{z\geq M\}} \widetilde{c}(\tau_n(r),z,X_{r}^{M })-\widetilde{c}(r,z,X_{r}^{M })\nu(dz)\vert^pdr\rightarrow0.
\]
Finally, using Burkholder's inequality,
\begin{eqnarray*}
E_3&\leq& {C}_p(T)\mathbb{E}\vert\int_{0}^{t}\int_{\{z \geq
M \}}\vert\widetilde{c}(\tau_{n}(r),\gamma_{n}(z),X_{\tau_{n}(r)}^{n,M})-\widetilde{c}(r,z,X_{r}^{M })\vert^2\nu(dz)\vert^\frac{p}{2}dr\\
&\leq& {C}_p(T)[R_n^3+(\bar{c}_2)^\frac{p}{2}\int_0^t\mathbb{E}\vert X_{\tau_{n}(r)}^{n,M }-X_{r}^{M }\vert^pdr],
\end{eqnarray*}
where (by (\ref{complement})),
\begin{eqnarray*}
R_n^3=\mathbb{E}\vert\int_{0}^{t}\int_{\{z \geq
M \}}\vert\widetilde{c}(\tau_{n}(r),\gamma_{n}(z),X_{r}^{M})-\widetilde{c}(r,z,X_{r}^{M })\vert^2\nu(dz)\vert^\frac{p}{2}dr\rightarrow0.
\end{eqnarray*}

Therefore, $\mathbb{E}\vert X_{t}^{n,M}-X_{t}^{M}\vert^p\leq{C_p}(T)[R_{n}+\int_0^t\mathbb{E}\vert X_{\tau_{n}(r)}^{n,M }-X_{r}^{M }\vert^pdr]$, with $R_{n}=R_n^1+R_n^2+R_n^3\rightarrow0$ as $n\rightarrow\infty$. One can easily check that $\mathbb{E}\vert X_{t}^{n,M}-X_{\tau_{n}(t)}^{n,M}\vert^p\rightarrow0$. Also there exists a constant $C_p(T)$ depending on $p$ and $T$ such that for any $n,M\in\mathbb{N}$ and any $t\in[0,T]$, $\mathbb{E}\vert X_{t}^{n,M}\vert^p\leq C_p(T)$ (see (\ref{extra}) for details). Then, by the dominated convergence theorem, these yield  $\int_0^t\mathbb{E}\vert X_{r}^{n,M}-X_{\tau_{n}(r)}^{n,M}\vert^pdr\rightarrow0$. So we have $\mathbb{E}\vert X_{t}^{n,M}-X_{t}^{M}\vert^p\leq{C_p}(T)[\widetilde{R_{n}}+\int_0^t\mathbb{E}\vert X_{r}^{n,M }-X_{r}^{M }\vert^pdr]$, with $\widetilde{R_{n}}\rightarrow0$ as $n\rightarrow\infty$. We conclude by using Gronwall's lemma. 
\end{proof}

\begin{remark} 
Some results on the convergence of the Euler scheme of a jump-diffusion can be found for example in $\cite{ref38},\cite{ref30}$. The special thing in our paper is that we deal with the space-time Brownian motion instead of the classical Brownian motion, and this is why we need to assume (\ref{complement}).
\end{remark}

Now we represent the jump's part of $(X_t^{n,M})_{t\in[0,T]}$ by means of compound Poisson processes. We recall that for each $k\in\mathbb{N}$, we denote by $T_i^k,i\in\mathbb{N}$ the jump times of a Poisson process $(J_t^k)_{t\in[0,T]}$ of parameter $m_k$, and we consider a sequence of independent random variables $Z_i^k\sim\mathbbm{1}_{I_k}(z)\frac{\nu(dz)}{m_k}, i\in\mathbb{N}$, which are independent of $J^k$ as well. Then we write
\begin{eqnarray}
X_{t}^{n,M } &=&x+\int_{0}^{t}\sum_{k=1}^{M-1}\int_{\{z \in
I_k\}}\widetilde{c}(\tau_{n}(r),z,X_{\tau_{n}(r)-}^{n.M})N_{\nu
}(dr,dz)   \nonumber\\
&&+\int_{0}^{t}b_{M}(\tau_{n}(r),X_{\tau_{n}(r)}^{n,M})dr+\int_{0}^{t}\int_{\{z\geq M\}}\widetilde{c}(\tau_{n}(r),\gamma_{n}(z),X_{\tau_{n}(r)}^{n,M})W_\nu(dr,dz)   \nonumber\\
&=&x+\sum_{k=1}^{M-1}\sum_{i=1}^{J_t^k}\widetilde{c}(\tau_{n}(T_i^k),Z_i^k,X_{\tau_{n}(T_i^k)-}^{n,M})   \nonumber\\
&&+\int_{0}^{t}b_{M}(\tau_{n}(r),X_{\tau_{n}(r)}^{n,M})dr+\int_{0}^{t}\int_{\{z\geq M\}}\widetilde{c}(\tau_{n}(r),\gamma_{n}(z),X_{\tau_{n}(r)}^{n,M})W_\nu(dr,dz)  . \label{3.03}
\end{eqnarray}%
So for every $t\in[0,T]$, $X_{t}^{n,M }$ is a simple functional.

\subsubsection{Preliminary estimates}
In order to estimate the Sobolev norms of the Euler scheme, we need the following preliminary lemmas.

\begin{lemma}
We fix $M\geq1$. Let $y:\Omega\times[0,T]\times[M,\infty)\rightarrow \mathbb{R}$ be a function which is piecewise constant with respect to both $t$ and $z$. We assume that $y_t(z)$ is progressively measurable with respect to $\mathcal{F}_t$ (defined in (\ref{Ft})), $y_t(z)\in\mathcal{S}$, and $\mathbb{E}(\int_{0}^{t}\int_{\{z\geq M \}}\left\vert
y_r(z)\right\vert ^{2}\nu (dz)dr)<\infty$. We denote $I_t({y})=\int_0^t\int_{\{z\geq M\}}{y}_r(z)W_\nu(dr,dz)$. Then for any $l\geq1,p\geq2$, there exists a constant $C_{l,p}(T)$ such that 
\[
a)\quad\mathbb{E}\vert I_t(y)\vert_{l}^p\leq C_{l,p}(T)\mathbb{E}\int_0^t(\int_{\{z\geq M\}}\vert y_r(z)\vert_{l}^2\nu(dz))^{\frac{p}{2}}dr,
\]
\begin{eqnarray*}
b)\quad\mathbb{E}\vert LI_t(y)\vert_l^p\leq C_{l,p}(T)[\mathbb{E}\int_0^t(\int_{\{z\geq M\}}\vert Ly_r(z)\vert_{l}^2\nu(dz))^{\frac{p}{2}}dr+\mathbb{E}\int_0^t(\int_{\{z\geq M\}}\vert y_r(z)\vert_{l}^2\nu(dz))^{\frac{p}{2}}dr].
\end{eqnarray*}
\end{lemma}

\begin{proof}
Proof of $a)$: Let $C_{l,p}(T)$ be a constant depending on $l,p$ and $T$ which may change from one line to another. For any $l\geq1$, we take $l_W\geq0$ and $l_Z\geq0$ such that $0<l_W+l_Z\leq l$.\\
It is easy to check that 
\[
D^{Z,l_Z}_{(k_1,i_1)\cdots(k_{l_Z},i_{l_Z})}I_t(y)=\int_0^t\int_{\{z\geq M\}}D^{Z,{l_Z}}_{(k_1,i_1)\cdots(k_{l_Z},i_{l_Z})} y_r(z)W_\nu(dr,dz).
\]
And by recurrence, one can show that
\[
D^{W,{l_W}}_{(s_1,z_1)\cdots(s_{l_W},z_{l_W})}I_t(y)=\int_0^t\int_{\{z\geq M\}}D^{W,{l_W}}_{(s_1,z_1)\cdots(s_{l_W},z_{l_W})} y_r(z)W_\nu(dr,dz)+\sum_{j=1}^{l_W}D^{W,{{l_W}-1}}_{\widehat{(s_j,z_j)}^{{l_W}-1}}y_{s_j}(z_j)\mathbbm{1}_{s_j\leq t},
\]
with 
\[
\widehat{(s_j,z_j)}^{{l_W}-1}:=(s_1,z_1)\cdots(s_{j-1},z_{j-1})(s_{j+1},z_{j+1})\cdots(s_{l_W},z_{l_W}).
\]
We denote 
\[
\bar{y}_r(z)(k_1,i_1,\cdots,k_{l_Z},i_{l_Z}):=D^{Z,{l_Z}}_{(k_1,i_1)\cdots(k_{l_Z},i_{l_Z})} y_r(z),\quad \bar{y}_r^{l_Z}(z):=D^{Z,{l_Z}}y_r(z)\in l_2^{\otimes l_Z}.
\]
Then $D^{Z,{l_Z}}I_t(y)=I_t(\bar{y}^{l_Z})$, and
\begin{eqnarray*}
D^{W,l_W}_{(s_1,z_1)\cdots(s_{l_W},z_{l_W})}D^{Z,l_Z}_{(k_1,i_1)\cdots(k_{l_Z},i_{l_Z})}I_t(y)&=&\int_0^t\int_{\{z\geq M\}}D^{W,l_W}_{(s_1,z_1)\cdots(s_{l_W},z_{l_W})} \bar{y}_r(z)(k_1,i_1,\cdots,k_{l_Z},i_{l_Z})W_\nu(dr,dz)\\
&+&\sum_{j=1}^{l_W}D^{W,{l_W-1}}_{\widehat{(s_j,z_j)}^{l_W-1}}\bar{y}_{s_j}(z_j)(k_1,i_1,\cdots,k_{l_Z},i_{l_Z})\mathbbm{1}_{s_j\leq t}.
\end{eqnarray*}
Let
${H}_{{l_Z},{l_W},T}=l_2^{\otimes {l_Z}}\otimes L^2([0,T]\times[M,\infty),Leb\times\nu)^{\otimes {l_W}}.$
We have
\begin{eqnarray*}
&&\vert D^{W,{l_W}}D^{Z,{l_Z}}I_t(y)\vert_{H_{l,\bar{l},T}}^2=\int_{[0,T]^{l_W}}\int_{[M,\infty)^{l_W}}\vert D^{W,{l_W}}_{(s_1,z_1)\cdots(s_{l_W},z_{l_W})} I_t(\bar{y}^{l_Z})\vert_{l_2^{\otimes {l_Z}}}^2\nu(dz_1)ds_1\cdots\nu(dz_{l_W})ds_{l_W}\\
&&\leq 2\vert\int_0^t\int_{\{z\geq M\}}D^{W,{l_W}} \bar{y}^{l_Z}_r(z)W_\nu(dr,dz)\vert_{H_{{l_Z},{l_W},T}}^2 + {l_W}2^{l_W}\int_0^t\int_{\{z\geq M\}}\vert D^{W,{{l_W}-1}} \bar{y}^{l_Z}_r(z)\vert_{H_{{l_Z},{l_W}-1,T}}^2\nu(dz)dr.
\end{eqnarray*}
Using Burkholder's inequality for Hilbert-space-valued martingales (see \cite{ref12} for example), we have
\begin{eqnarray}
\mathbb{E}\vert D^{W,{l_W}}D^{Z,{l_Z}}I_t(y)\vert_{H_{{l_Z},{l_W},T}}^p&\leq& C_{l,p}(T)[\mathbb{E}\int_0^t(\int_{\{z\geq M\}}\vert D^{W,{l_W}}D^{Z,{l_Z}}y_r(z)\vert_{H_{{l_Z},{l_W},T}}^2\nu(dz))^{\frac{p}{2}}dr  \nonumber\\
&+&\mathbb{E}\int_0^t(\int_{\{z\geq M\}}\vert D^{W,{{l_W}-1}}D^{Z,{l_Z}}y_r(z)\vert_{H_{{l_Z},{l_W}-1,T}}^2\nu(dz))^{\frac{p}{2}}dr].   \label{ap1}
\end{eqnarray}

We recall that for $F\in\mathcal{D}_\infty$, we have $\vert D^{W,{l_W}}D^{Z,{l_Z}}F\vert_{H_{{l_Z},{l_W},T}}\leq \vert F\vert_{{l_Z}+{l_W}}$ (see the definition in (\ref{norm})). Then (\ref{ap1}) gives
\begin{eqnarray}
\mathbb{E}\vert I_t(y)\vert_{1,l}^p\leq C_{l,p}(T)\sum\limits_{{l_Z}+{l_W}\leq l}\mathbb{E}\vert D^{W,{l_W}}D^{Z,{l_Z}}I_t(y)\vert_{H_{{l_Z},{l_W},T}}^p\leq {C}_{l,p}(T)\mathbb{E}\int_0^t(\int_{\{z\geq M\}}\vert y_r(z)\vert_{l}^2\nu(dz))^{\frac{p}{2}}dr. \label{auxD1}
\end{eqnarray}

Finally, using Burkholder's inequality, we have 
\begin{eqnarray}
\mathbb{E}\vert I_t(y)\vert^p\leq {C}_{l,p}(T)\mathbb{E}\int_0^t(\int_{\{z\geq M\}}\vert y_r(z)\vert^2\nu(dz))^{\frac{p}{2}}dr.  \label{auxD2}
\end{eqnarray}
So $a)$ is proved.

\bigskip\bigskip

Proof of $b)$: We first show that 
\begin{eqnarray}
LI_t(y)=I_t(Ly)+I_t(y).  \label{ap2}
\end{eqnarray}
We denote
\[
I_k^t(f_k)=k!\int_0^t\int_0^{s_1}\cdots\int_0^{s_{k-1}}\int_{[M,+\infty)^k}f_k(s_1,\cdots,s_k,z_1,\cdots,z_k)W_\nu(ds_k,dz_k)\cdots W_\nu(ds_1,dz_1)
\]
the multiple stochastic integral for a deterministic function $f_k$, which is square integrable with respect to $(\nu(dz)ds)^{\otimes k}$ and is symmetric with respect to the time variation $(s_1,\cdots,s_k)$ for each fixed $(z_1,\cdots,z_k)$. Notice that $L^ZI_k^t(f_k)=0$ and $L^WI_k^t(f_k)=kI_k^t(f_k)$. So, $LI_k^t(f_k)=kI_k^t(f_k)$.
Then by the duality relation (\ref{0.01}), 
\begin{eqnarray}
\mathbb{E}(I_k^t(f_k)L(I_t(y)))=\mathbb{E}(I_t(y)\times LI_k^t(f_k))=k\mathbb{E}(I_t(y)\times I_k^t(f_k)).  \label{ap3}
\end{eqnarray}
On the other hand, using the isometry property and the duality relation,
\begin{eqnarray*}
&&\mathbb{E}(I_k^t(f_k)\times I_t(Ly))=k\mathbb{E}\int_0^t\int_{\{z\geq M\}}I_{k-1}^r(f_k(r,z,\cdot))Ly_r(z)\nu(dz)dr\\
&&=k\int_0^t\int_{\{z\geq M\}}\mathbb{E}[y_r(z)\times LI_{k-1}^r(f_k(r,z,\cdot))]\nu(dz)dr=k(k-1)\mathbb{E}\int_0^t\int_{\{z\geq M\}}y_r(z)I_{k-1}^r(f_k(r,z,\cdot))\nu(dz)dr\\
&&=k(k-1)\mathbb{E}(I_t(y)\times\int_0^t\int_{\{z\geq M\}}I_{k-1}^r(f_k(r,z,\cdot))W_\nu(dr,dz))\\
&&=(k-1)\mathbb{E}(I_t(y)\times I_{k}^t(f_k)) .
\end{eqnarray*}Combining this with (\ref{ap3}), we get
\begin{eqnarray}
\mathbb{E}(I_k^t(f_k)(I_t(y)+I_t(Ly)))=k\mathbb{E}(I_k^t(f_k)I_t(y))=\mathbb{E}[I_k^t(f_k)\times LI_t(y)]. 
\end{eqnarray}

Since every element in $L^2(W)$ (defined by (\ref{ap})) can be represented as the direct sum of multiple stochastic integrals, we have for any $F\in L^2(W)$,
\begin{eqnarray}
\mathbb{E}[FLI_t(y)] =\mathbb{E}[F(I_t(Ly)+I_t(y))] .  \label{ap5}
\end{eqnarray}

For $G\in L^2(N)$, one has $L^WG=0$ and $L^ZG\in L^2(N)$. Then by using duality and (\ref{LWN}), 
\[
\mathbb{E}[GLI_t(y)]=\mathbb{E}[I_t(y)LG]=\mathbb{E}[I_t(y)L^ZG]=0,
\]and by (\ref{LWN}),
\[
\mathbb{E}[G(I_t(Ly)+I_t(y))]=0.
\]
So, \begin{eqnarray}
\mathbb{E}[GLI_t(y)] =\mathbb{E}[G(I_t(Ly)+I_t(y))] .  \label{ap6}
\end{eqnarray}
Combining (\ref{ap5}) and (\ref{ap6}), for any $\Tilde{G}\in L^2(W)\otimes L^2(N)$, we have
$\mathbb{E}[\Tilde{G}LI_t(y)] =\mathbb{E}[\Tilde{G}(I_t(Ly)+I_t(y))]$,
which proves (\ref{ap2}).\\
Then, by \textbf{Lemma 5.2} $a)$, 
\begin{eqnarray*}
\mathbb{E}\vert LI_t(y)\vert_{l}^p&\leq&2^{p-1}(\mathbb{E}\vert\int_0^t\int_{\{z\geq M\}}Ly_r(z)W_\nu(dr,dz)\vert_l^p+\mathbb{E}\vert\int_0^t\int_{\{z\geq M\}}y_r(z)W_\nu(dr,dz)\vert_l^p)\\
&\leq& C_{l,p}(T)[\mathbb{E}\int_0^t(\int_{\{z\geq M\}}\vert Ly_r(z)\vert_{l}^2\nu(dz))^{\frac{p}{2}}dr+\mathbb{E}\int_0^t(\int_{\{z\geq M\}}\vert y_r(z)\vert_{l}^2\nu(dz))^{\frac{p}{2}}dr].
\end{eqnarray*}
\end{proof}

We will also need the following lemma from $\cite{ref4}$ (Lemma\ 8 and Lemma\ 10), which is a consequence of the chain rule for $D^q$ and $L$.
\begin{lemma}
Let $F\in\mathcal{S}^d$. For every $l\in\mathbb{N},$ if $\phi: \mathbb{R}^d\rightarrow\mathbb{R}$ is a $C^{l}(\mathbb{R}^d)$ function ($l-$times differentiable function), then
\[
a)\quad \vert\phi(F)\vert_{1,l}\leq\vert\nabla\phi(F)\vert\vert F\vert_{1,l}+C_l\sup_{2\leq\vert\beta\vert\leq l}\vert\partial^\beta\phi(F)\vert\vert F\vert_{1,l-1}^{l}.
\]
If $\phi\in C^{l+2}(\mathbb{R}^d)$, then
\[
b)\quad \vert L\phi(F)\vert_{l}\leq\vert\nabla\phi(F)\vert\vert LF\vert_{l}+C_l\sup_{2\leq\vert\beta\vert\leq l+2}\vert\partial^\beta\phi(F)\vert(1+\vert F\vert_{l+1}^{l+2})(1+\vert LF\vert_{l-1}).
\]
For $l=0$, we have
\[
c)\quad \vert L\phi(F)\vert\leq\vert\nabla\phi(F)\vert\vert LF\vert+\sup_{\vert\beta\vert=2}\vert\partial^\beta\phi(F)\vert\vert F\vert_{1,1}^{2}.
\]
\end{lemma}

We finish this section with a first estimate concerning the operator $L$.
\begin{lemma}
Under the \textbf{Hypothesis 2.4} (either 2.4(a) or 2.4(b)), for every $p\geq2,\Tilde{p}\geq1,l\geq0$, there exists a constant $C_{l,p,\Tilde{p}}(T)$ such that
\begin{eqnarray}
\sup_{M\in\mathbb{N}}\mathbb{E} (\sum_{k=1}^{M-1}\sum_{i=1}^{J_t^k}\bar{c}(Z_i^k)\vert LZ_i^k\vert_l^{\Tilde{p}})^{p}\leq C_{l,p,\Tilde{p}}(T). \label{LZTrue}
\end{eqnarray}
\end{lemma}
\begin{proof}
We notice that (with $\psi_k$ given in (\ref{3.19*})),
$LZ_i^k=\xi_i^k(\ln{\psi_k})^{\prime}(V_i^k)$
and $D^{W,l}LZ^k_i=0$. Moreover,
\[D^{Z,l}_{(r_1,m_1)\cdots(r_l,m_l)}LZ_i^k\\
=\prod\limits_{j=1}^l(\delta_{r_jk}\delta_{m_ji})\xi_i^k(\ln{\psi_k})^{(l+1)}(V_i^k),
\]with $\delta_{rk}$ the Kronecker delta,
so that \begin{eqnarray}
\vert LZ_i^k\vert_l=\xi_i^k\sum\limits_{0\leq\Tilde{l}\leq l}\vert(\ln{\psi_k})^{(\Tilde{l}+1)}(V_i^k)\vert. \label{LZtrue}
\end{eqnarray}It follows that
\begin{eqnarray*}
\mathbb{E} (\sum_{k=1}^{M-1}\sum_{i=1}^{J_t^k}\bar{c}(Z_i^k)\vert LZ_i^k\vert_l^{\Tilde{p}})^{p}\leq C_{l,p,\Tilde{p}}\sum\limits_{0\leq\Tilde{l}\leq l}\mathbb{E} (\sum_{k=1}^{M-1}\sum_{i=1}^{J_t^k}\bar{c}(Z_i^k)\xi_i^k\vert(\ln{\psi_k})^{(\Tilde{l}+1)}(V_i^k)\vert^{\Tilde{p}})^{p}.
\end{eqnarray*}
Since $\bar{c}(Z_i^k)\xi_i^k=\bar{c}(V_i^k)\xi_i^k$, we may replace $Z_i^k$ by $V_i^k$ in the right hand side of the above estimate. This gives
\[
C_{l,p,\Tilde{p}}\sum\limits_{0\leq\Tilde{l}\leq l}\mathbb{E} (\sum_{k=1}^{M-1}\sum_{i=1}^{J_t^k}\bar{c}(V_i^k)\xi_i^k\vert(\ln{\psi_k})^{(\Tilde{l}+1)}(V_i^k)\vert^{\Tilde{p}})^{p}=C_{l,p,\Tilde{p}}\sum\limits_{0\leq\Tilde{l}\leq l}\mathbb{E} \vert\int_0^t\int_{[1,M)}\int_{\{0,1\}}\bar{c}(v)\xi\vert(\ln{\bar{\psi}})^{(\Tilde{l}+1)}(v)\vert^{\Tilde{p}}\Lambda(ds,d\xi,dv)\vert^{p},
\]where $\bar{\psi}(v):=\sum\limits_{k=1}^\infty\mathbbm{1}_{I_k}(v)\psi(v-(k+\frac{1}{2}))$ and $\Lambda$ is a Poisson point measure on $\{0,1\}\times[1,\infty)$ with compensator 
\[
\widehat{\Lambda}(ds,d\xi,dv)=\sum_{k=1}^\infty[\frac{\psi(v-(k+\frac{1}{2}))}{m(\psi)}\mathbbm{1}_{I_k}(v)dv\times b(v,d\xi)]ds,
\]with $b(v,d\xi)$ the Bernoulli probability measure on $\{0,1\}$ with parameter $\varepsilon_km(\psi)$, if $v\in I_k$.
Then by compensating $\Lambda$ and using Burkholder's inequality (the same proof as for (\ref{Burk})),
\begin{eqnarray}
&&C_{l,p,\Tilde{p}}\sum\limits_{0\leq\Tilde{l}\leq l}\mathbb{E} \vert\int_0^t\int_{[1,M)}\int_{\{0,1\}}\bar{c}(v)\xi\vert(\ln{\bar{\psi}})^{(\Tilde{l}+1)}(v)\vert^{\Tilde{p}}\Lambda(ds,d\xi,dv)\vert^{p} \nonumber\\
&&\leq C_{l,p,\Tilde{p}}(T)\sum\limits_{0\leq\Tilde{l}\leq l}[ (\int_0^t\int_{[1,M)\times\{0,1\}}\vert\bar{c}(v)\vert^2\xi\vert(\ln{\bar{\psi}})^{(\Tilde{l}+1)}(v)\vert^{2\Tilde{p}}\widehat{\Lambda}(ds,d\xi,dv))^{\frac{p}{2}} \nonumber\\
&&+ \int_0^t\int_{[1,M)\times\{0,1\}}\vert\bar{c}(v)\vert^p\xi\vert(\ln{\bar{\psi}})^{(\Tilde{l}+1)}(v)\vert^{p\Tilde{p}}\widehat{\Lambda}(ds,d\xi,dv)+ \vert\int_0^t\int_{[1,M)\times\{0,1\}}\bar{c}(v)\xi\vert(\ln{\bar{\psi}})^{(\Tilde{l}+1)}(v)\vert^{\Tilde{p}}\widehat{\Lambda}(ds,d\xi,dv)\vert^{p}]. \label{LZtemp} \quad\quad
\end{eqnarray}
We notice that by (\ref{3.19}), \begin{eqnarray*}
&&\int_0^t\int_{[1,M)\times\{0,1\}}\vert\bar{c}(v)\vert^p\xi\vert(\ln{\bar{\psi}})^{(\Tilde{l}+1)}(v)\vert^{p\Tilde{p}}\widehat{\Lambda}(ds,d\xi,dv)=t\sum_{k=1}^{M-1}\int_{I_k}\varepsilon_km(\psi)\vert\bar{c}(v)\vert^p\vert(\ln{{\psi_k}})^{(\Tilde{l}+1)}(v)\vert^{p\Tilde{p}}\frac{\psi_k(v)}{m(\psi)}dv\\
&&\leq C_{\Tilde{l},p,\Tilde{p}}(T)\sum_{k=1}^{M-1}\varepsilon_k\int_{I_k}\vert\bar{c}(v)\vert^pdv.
\end{eqnarray*}Similar upper bounds hold for the two other terms in the right hanf side of (\ref{LZtemp}), so (\ref{LZtemp}) is upper bounded by
\begin{eqnarray}
C_{l,p,\Tilde{p}}(T)[(\sum_{k=1}^{M-1}\varepsilon_k\int_{I_k}\vert\bar{c}(v)\vert^2dv)^{\frac{p}{2}}+\sum_{k=1}^{M-1}\varepsilon_k\int_{I_k}\vert\bar{c}(v)\vert^pdv+(\sum_{k=1}^{M-1}\varepsilon_k\int_{I_k}\vert\bar{c}(v)\vert dv)^p].\label{LZTemp}
\end{eqnarray}If we assume the \textbf{Hypothesis 2.4} (a), then we have $\varepsilon_k=\varepsilon_\ast/(k+1)^{1-\alpha_0}$, with $\alpha_0$ given in (\ref{aLfault}). So the above term is less than 
\[
C_{l,p,\Tilde{p}}(T)[(\int_{1}^\infty\frac{\vert\bar{c}(v)\vert^2}{v^{1-\alpha_0}}dv)^{\frac{p}{2}}+\int_{1}^\infty\frac{\vert\bar{c}(v)\vert^p}{v^{1-\alpha_0}}dv+(\int_{1}^\infty\frac{\vert\bar{c}(v)\vert}{v^{1-\alpha_0}} dv)^p], 
\]which is upper bounded by a constant $C_{l,p,\Tilde{p}}(T)$ thanks to (\ref{aLfault}).
On the other hand, if we assume the \textbf{Hypothesis 2.4} (b), then $\varepsilon_k=\varepsilon_\ast/(k+1)$. So (\ref{LZTemp}) is upper bounded by a constant $C_{l,p,\Tilde{p}}(T)$ thanks to (\ref{bLfault}).

\end{proof}

\subsubsection{Estimations of $\Vert X_t^{n,M}\Vert_{L,l,p}$}
In this section, our aim is to prove the following lemma.
\begin{lemma}
Under the \textbf{Hypothesis 2.1} with $q^\ast\geq2$ and \textbf{Hypothesis 2.4} (either 2.4(a) or 2.4(b)), for all $p\geq2, 0\leq l\leq q^\ast$, there exists a constant $C_{l,p}(T)$ depending on $l,p,x$ and $T$, such that 
\begin{eqnarray}
a)\quad \sup\limits_{n}\sup\limits_{M}\mathbb{E}\vert X_t^{n,M}\vert_{l}^p\leq C_{l,p}(T), \label{EuD}
\end{eqnarray}and for $0\leq l\leq q^\ast-2$,
\begin{eqnarray}
b)\quad \sup\limits_{n}\sup\limits_{M}\mathbb{E}\vert LX_t^{n,M}\vert_{l}^p\leq C_{l,p}(T). \label{EuL}
\end{eqnarray}

\end{lemma}
\begin{proof}
In the following proof, $C_{l,p}(T)$ will be a constant which depends on $l,p,x$ and $T$, and which may change from a line to another. $q^\ast\geq2$ is fixed throughout the proof.

\bigskip

$a)$ We prove (\ref{EuD}) for $0\leq l\leq q^\ast$ by recurrence on $l$.

\textbf{Step 1} For $l=0$, using Burkholder's inequality, \textbf{Hypothesis 2.1} and (\ref{complement*}),
\begin{eqnarray}
\mathbb{E}\vert X_t^{n,M}\vert^p&\leq& C_{0,p}(T)[x^p+\mathbb{E}\vert\int_0^tb_M(\tau_n(r),X_{\tau_n(r)}^{n,M})dr\vert^p+\mathbb{E}\vert\int_0^t\int_{\{z\geq M\}}\widetilde{c}(\tau_n(r),\gamma_n(z),X_{\tau_n(r)}^{n,M})W_\nu(dr,dz)\vert^p \nonumber\\
&+&\mathbb{E}\vert\int_0^t\int_{[1,M)}\widetilde{c}(\tau_n(r),z,X_{\tau_n(r)-}^{n,M})N_\nu(dr,dz)\vert^p] \nonumber\\
&\leq& C_{0,p}(T)[1+\mathbb{E}\int_0^t\vert \int_{\{z\geq M\}}\widetilde{c}(\tau_n(r),z,X_{\tau_n(r)}^{n,M})\nu(dz)\vert^pdr\nonumber\\
&+&\mathbb{E}\int_0^t(\int_{\{z\geq M\}}\vert\widetilde{c}(\tau_n(r),\gamma_n(z),X_{\tau_n(r)}^{n,M})\vert^2\nu(dz))^{\frac{p}{2}}dr+\mathbb{E}\int_0^t(\int_{[1,M)}\vert\widetilde{c}(\tau_n(r),z,X_{\tau_n(r)-}^{n,M})\vert^2\nu(dz))^{\frac{p}{2}}dr\nonumber\\
&+&\mathbb{E}\int_0^t\int_{[1,M)}\vert\widetilde{c}(\tau_n(r),z,X_{\tau_n(r)-}^{n,M})\vert^p\nu(dz)dr+\mathbb{E}\int_0^t\vert \int_{[1,M)}\widetilde{c}(\tau_n(r),z,X_{\tau_n(r)-}^{n,M})\nu(dz)\vert^pdr] \nonumber\\
&\leq&C_{0,p}(T). \label{extra}
\end{eqnarray}

\textbf{Step 2} Now we assume that (\ref{EuD}) holds for $l-1$, with $l\geq1$ and for every $p\geq2$, and we prove that it holds for $l$ and for every $p\geq2$. 
We write $\mathbb{E}\vert X_t^{n,M}\vert_{l}^p\leq C_{l,p}(T)(A_1+A_2+A_3)$, with 
\[
A_1=\mathbb{E}\vert\int_0^tb_M(\tau_n(r),X_{\tau_n(r)}^{n,M})dr\vert_{l}^p,
\]
\[
A_2=\mathbb{E}\vert\int_0^t\int_{\{z\geq M\}}\widetilde{c}(\tau_n(r),\gamma_n(z),X_{\tau_n(r)}^{n,M})W_\nu(dr,dz)\vert_{l}^p,
\]
\[
A_3=\mathbb{E}\vert\int_0^t\int_{[1,M)}\widetilde{c}(\tau_n(r),z,X_{\tau_n(r)-}^{n,M})N_\nu(dr,dz)\vert_{l}^p.
\]

We notice that by \textbf{Hypothesis 2.1}, 
$\Vert b_M\Vert_{l,\infty}\leq \bar{c}_1.$
Then using \textbf{Lemma 5.3} $a)$ and the recurrence hypothesis, we get
\begin{eqnarray}
A_1&\leq& {C}_{l,p}(T)\mathbb{E}\int_0^t\vert b_M(\tau_n(r),X_{\tau_n(r)}^{n,M})\vert_{l}^pdr\nonumber \\
&\leq& {C}_{l,p}(T)[(\bar{c}_1)^p+\mathbb{E}\int_0^t\vert\partial_x b_M(\tau_n(r),X_{\tau_n(r)}^{n,M})\vert^p\vert X_{\tau_n(r)}^{n,M}\vert_{1,l}^{p}dr\nonumber\\
&+&\mathbb{E}\int_0^t\sup\limits_{2\leq\vert\beta\vert\leq l}\vert\partial_x^\beta b_M(\tau_n(r),X_{\tau_n(r)}^{n,M})\vert^p\vert X_{\tau_n(r)}^{n,M}\vert_{1,l-1}^{lp}dr]\nonumber\\
&\leq&{C}_{l,p}(T)[1+\int_0^t\mathbb{E}\vert X_{\tau_n(r)}^{n,M}\vert_{l}^{p}dr]. \label{A1}
\end{eqnarray}

Next, we estimate $A_2$. By \textbf{Hypothesis 2.1}, for every $n$, $\Vert\widetilde{c}(\tau_n(r),\gamma_n(z),\cdot)\Vert_{l,\infty}\leq \vert\bar{c}(\gamma_n(z))\vert$.
Then using \textbf{Lemma 5.2} $a)$, \textbf{Lemma 5.3} $a)$, (\ref{euler_hyp}) and the recurrence hypothesis, we get
\begin{eqnarray}
A_2&\leq& {C}_{l,p}(T)[\mathbb{E}\int_0^t(\int_{\{z\geq M\}}\vert\widetilde{c}(\tau_n(r),\gamma_n(z),X_{\tau_n(r)}^{n,M})\vert_{l}^2\nu(dz))^{\frac{p}{2}}dr\nonumber\\
&\leq& {C}_{l,p}(T)[\mathbb{E}\int_0^t(\int_{\{z\geq M\}}\vert\partial_{x}\widetilde{c}(\tau_n(r),\gamma_n(z),X_{\tau_n(r)}^{n,M})\vert^2\vert X_{\tau_n(r)}^{n,M}\vert_{1,l}^2\nu(dz))^{\frac{p}{2}}dr\nonumber\\
&+&\mathbb{E}\int_0^t(\int_{\{z\geq M\}}\sup\limits_{2\leq\vert\beta\vert\leq l}\vert\partial_{x}^\beta \widetilde{c}(\tau_n(r),\gamma_n(z),X_{\tau_n(r)}^{n,M})\vert^2\vert X_{\tau_n(r)}^{n,M}\vert_{1,l-1}^{2l}\nu(dz))^{\frac{p}{2}}dr\nonumber\\
&+&\mathbb{E}\int_0^t(\int_{\{z\geq M\}}\vert\widetilde{c}(\tau_n(r),\gamma_n(z),X_{\tau_n(r)}^{n,M})\vert^2\nu(dz))^{\frac{p}{2}}dr\nonumber\\
&\leq&{C}_{l,p}(T)[1+\int_0^t\mathbb{E}\vert X_{\tau_n(r)}^{n,M}\vert_{l}^{p}dr]. \label{A2}
\end{eqnarray}

Finally we estimate $A_3$.
We notice that 
$D^{Z}_{(r,m)}Z_i^k=\xi_i^k\delta_{rk}\delta_{mi}$, $D^W_{(s,z)}Z_i^k=0,$
and for $l\geq2$, \\
$D^{Z,W,l}_{(r_1,m_1)\cdots(r_l,m_l),(s_1,z_1)\cdots(s_l,z_l)}Z_i^k=0$. So we have $\vert Z_i^k\vert_{1,l}^p= \vert\xi_i^k\vert^p\leq1$.
By \textbf{Lemma 5.3} $a)$ for $d=2$, \textbf{Hypothesis 2.1}, for any $k,i\in\mathbb{N}$,
\begin{eqnarray*}
&&\vert\widetilde{c}(\tau_n(T_i^k),Z_i^k,X^{n,M}_{\tau_n(T_i^k)-})\vert_{l}\leq \vert \bar{c}(Z_i^k)\vert\\
&&+(\vert\partial_{z}\widetilde{c}(\tau_n(T_i^k),Z_i^k,X_{\tau_n(T_i^k)-}^{n,M})\vert+\vert\partial_{x}\widetilde{c}(\tau_n(T_i^k),Z_i^k,X_{\tau_n(T_i^k)-}^{n,M})\vert)(\vert Z_i^k\vert_{1,l}+\vert X_{\tau_n(T_i^k)-}^{n,M}\vert_{1,l})\\
&&+{C}_{l,p}(T)\sup\limits_{2\leq\vert\beta_1+\beta_2\vert\leq l}(\vert\partial_{z}^{\beta_2}\partial_{x}^{\beta_1}\widetilde{c}(\tau_n(T_i^k),Z_i^k,X_{\tau_n(T_i^k)-}^{n,M})\vert)(\vert Z_i^k\vert_{1,l-1}^{l}+\vert X_{\tau_n(T_i^k)-}^{n,M}\vert_{1,l-1}^{l})\\
&&\leq{C}_{l,p}(T)\bar{c}(Z_i^k)(1+\vert X_{\tau_n(T_i^k)-}^{n,M}\vert_{l}+\vert X_{\tau_n(T_i^k)-}^{n,M}\vert_{l-1}^{l}).
\end{eqnarray*}
It follows that 
\begin{eqnarray}
A_3&\leq& \mathbb{E}(\sum_{k=1}^{M-1}\sum_{i=1}^{J^k_t}\vert\widetilde{c}(\tau_n(T_i^k),Z_i^k,X^{n,M}_{\tau_n(T_i^k)-})\vert_{l})^p\leq {C}_{l,p}(T)\mathbb{E}(\sum_{k=1}^{M-1}\sum_{i=1}^{J^k_t}\bar{c}(Z_i^k)(1+\vert X_{\tau_n(T_i^k)-}^{n,M}\vert_{l}+\vert X_{\tau_n(T_i^k)-}^{n,M}\vert_{l-1}^{l}))^p \nonumber\\
&=&{C}_{l,p}(T)\mathbb{E}(\int_0^t\int_{[1,M)}{\bar{c}}(z)(1+\vert X_{\tau_n(r)-}^{n,M}\vert_{l}+\vert X_{\tau_n(r)-}^{n,M}\vert_{l-1}^{l})N_\nu(dr,dz))^p\nonumber\\
&\leq& {C}_{l,p}(T)[1+\int_0^t\mathbb{E}\vert X_{\tau_n(r)}^{n,M}\vert_{l}^pdr],  \label{A3}
\end{eqnarray}
where the last inequality is obtained by using (\ref{Burk}) and recurrence hypothesis.
Then combining (\ref{A1}),(\ref{A2}) and (\ref{A3}), 
\begin{eqnarray}
\mathbb{E}\vert X_t^{n,M}\vert_{l}^p\leq {C}_{l,p}(T)[1+\int_0^t\mathbb{E}\vert X_{\tau_n(r)}^{n,M}\vert_{l}^pdr]. \label{gronD}
\end{eqnarray}
So $\mathbb{E}\vert X_{\tau_n(t)}^{n,M}\vert_{l}^p\leq {C}_{l,p}(T)[1+\int_0^{\tau_n(t)}\mathbb{E}\vert X_{\tau_n(r)}^{n,M}\vert_{l}^pdr]\leq {C}_{l,p}(T)[1+\int_0^{t}\mathbb{E}\vert X_{\tau_n(r)}^{n,M}\vert_{l}^pdr]$. We denote temporarily $g(t)=\mathbb{E}\vert X_{\tau_n(t)}^{n,M}\vert_{l}^p$, then we have $g(t)\leq {C}_{l,p}(T)[1+\int_0^{t}g(r)dr]$.
By Gronwall's lemma, $g(t)\leq {C}_{l,p}(T)e^{T{C}_{l,p}(T)}$,  which means that
\[
\mathbb{E}\vert X_{\tau_n(t)}^{n,M}\vert_{l}^p\leq {C}_{l,p}(T)e^{T{C}_{l,p}(T)}.
\]
Substituting into (\ref{gronD}), we conclude that 
\begin{eqnarray}
\sup_{n,M}\mathbb{E}\vert X_t^{n,M}\vert_{l}^p\leq C_{l,p}(T). \label{D*}
\end{eqnarray}

As a summary of the recurrence argument, we remark that the uniform bound in $n,M$ of the operator $D$ for $l=0$ is due to the \textbf{Hypothesis 2.1}, and it propagates to larger $l$ thanks to \textbf{Lemma 5.3} $a)$.

\bigskip

$b)$ Now we prove (\ref{EuL}) for $0\leq l\leq q^\ast-2$, by recurrence on $l$. 

\textbf{Step 1} One has to check that (\ref{EuL}) holds for $l=0$.  The proof is analogous to that in the following \textbf{Step 2}, but simpler. It is done by using \textbf{Lemma 5.3} $c)$, (\ref{ap2}), Burkholder's inequality, \textbf{Hypothesis 2.1,2.4}, (\ref{euler_hyp}), (\ref{EuD}) and Gronwall's lemma. So we skip it.

\textbf{Step 2} Now we assume that (\ref{EuL}) holds for $l-1$, with $l\geq1$ and for any $p\geq2$ and we prove that it holds for $l$ and for any $p\geq2$. We write $\mathbb{E}\vert LX_t^{n,M}\vert_{l}^p\leq C_{l,p}(T)(B_1+B_2+B_3)$, with 
\[
B_1=\mathbb{E}\vert L\int_0^tb_M(\tau_n(r),X_{\tau_n(r)}^{n,M})dr\vert_{l}^p,
\]
\[
B_2=\mathbb{E}\vert L\int_0^t\int_{\{z\geq M\}}\widetilde{c}(\tau_n(r),\gamma_n(z),X_{\tau_n(r)}^{n,M})W_\nu(dr,dz)\vert_{l}^p,
\]
\[
B_3=\mathbb{E}\vert L\int_0^t\int_{[1,M)}\widetilde{c}(\tau_n(r),z,X_{\tau_n(r)-}^{n,M})N_\nu(dr,dz)\vert_{l}^p.
\]
Using \textbf{Lemma 5.3} $b)$, \textbf{Hypothesis 2.1}, the recurrence hypothesis and (\ref{EuD}), we get
\begin{eqnarray}
B_1&\leq& {C}_{l,p}(T)\mathbb{E}\int_0^t\vert Lb_M(\tau_n(r),X_{\tau_n(r)}^{n,M})\vert_{l}^pdr \nonumber\\
&\leq& {C}_{l,p}(T)[\mathbb{E}\int_0^t\vert\partial_x b_M(\tau_n(r),X_{\tau_n(r)}^{n,M})\vert^p\vert LX_{\tau_n(r)}^{n,M}\vert_{l}^{p}dr \nonumber\\
&+&\mathbb{E}\int_0^t\sup\limits_{2\leq\vert\beta\vert\leq l+2}\vert\partial_x^\beta b_M(\tau_n(r),X_{\tau_n(r)}^{n,M})\vert^p(1+\vert X_{\tau_n(r)}^{n,M}\vert_{l+1}^{(l+2)p})(1+\vert LX_{\tau_n(r)}^{n,M}\vert_{l-1}^{p})dr] \nonumber\\
&\leq&{C}_{l,p}(T)[1+\int_0^t\mathbb{E}\vert LX_{\tau_n(r)}^{n,M}\vert_{l}^{p}dr]. \label{B1}
\end{eqnarray}

Then by \textbf{Lemma 5.2} $b)$, we get
\begin{eqnarray*}
B_2&\leq& C_{l,p}(T)[\mathbb{E}\int_0^t(\int_{\{z\geq M\}}\vert L\widetilde{c}(\tau_n(r),\gamma_n(z),X_{\tau_n(r)}^{n,M})\vert_{l}^2\nu(dz))^{\frac{p}{2}}dr\\
&+&\mathbb{E}\int_0^t(\int_{\{z\geq M\}}\vert\widetilde{c}(\tau_n(r),\gamma_n(z),X_{\tau_n(r)}^{n,M})\vert_{l}^2\nu(dz))^{\frac{p}{2}}dr]\\
&:=&C_{l,p}(T)[B_{2,1}+B_{2,2}].
\end{eqnarray*}
As a consequence of \textbf{Lemma 5.3} $b)$, we have
\begin{eqnarray*}
B_{2,1}&\leq& C_{l,p}(T)[\mathbb{E}\int_0^t(\int_{\{z\geq M\}}\vert\partial_{x}\widetilde{c}(\tau_n(r),\gamma_n(z),X_{\tau_n(r)}^{n,M})\vert^2\vert LX_{\tau_n(r)}^{n,M}\vert_{l}^2\nu(dz))^{\frac{p}{2}}dr\\
&+&\mathbb{E}\int_0^t(\int_{\{z\geq M\}}\sup\limits_{2\leq\vert\beta\vert\leq l+2}\vert\partial_{x}^\beta \widetilde{c}(\tau_n(r),\gamma_n(z),X_{\tau_n(r)}^{n,M})\vert^2(1+\vert X_{\tau_n(r)}^{n,M}\vert_{l+1}^{2(l
+2)})(1+\vert LX_{\tau_n(r)}^{n,M}\vert_{l-1}^{2})\nu(dz))^{\frac{p}{2}}dr].
\end{eqnarray*}
And using \textbf{Lemma 5.3} $a)$, 
\begin{eqnarray*}
B_{2,2}&\leq& C_{l,p}(T)[\mathbb{E}\int_0^t(\int_{\{z\geq M\}}\vert\widetilde{c}(\tau_n(r),\gamma_n(z),X_{\tau_n(r)}^{n,M})\vert^2\nu(dz))^{\frac{p}{2}}dr\\
&+&\mathbb{E}\int_0^t(\int_{\{z\geq M\}}\vert\partial_{x}\widetilde{c}(\tau_n(r),\gamma_n(z),X_{\tau_n(r)}^{n,M})\vert^2\vert X_{\tau_n(r)}^{n,M}\vert_{1,l}^2\nu(dz))^{\frac{p}{2}}dr\\
&+&\mathbb{E}\int_0^t(\int_{\{z\geq M\}}\sup\limits_{2\leq\vert\beta\vert\leq l}\vert\partial_{x}^\beta \widetilde{c}(\tau_n(r),\gamma_n(z),X_{\tau_n(r)}^{n,M})\vert^2\vert X_{\tau_n(r)}^{n,M}\vert_{1,l-1}^{2l}\nu(dz))^{\frac{p}{2}}dr].
\end{eqnarray*}
Then by \textbf{Hypothesis 2.1}, (\ref{euler_hyp}), (\ref{EuD}) and the recurrence hypothesis,
\begin{eqnarray}
B_2\leq  C_{l,p}(T)[1+\int_0^t\mathbb{E}\vert LX_{\tau_n(r)}^{n,M}\vert_{l}^{p}dr].  \label{B2}
\end{eqnarray}

Now we estimate $B_3$.
By \textbf{Lemma 5.3} $b)$ for $d=2$, \textbf{Hypothesis 2.1}, for any $k,i\in\mathbb{N}$,
\begin{eqnarray*}
&&\vert L\widetilde{c}(\tau_n(T_i^k),Z_i^k,X^{n,M}_{\tau_n(T_i^k)-})\vert_{l}\leq(\vert\partial_{z}\widetilde{c}(\tau_n(T_i^k),Z_i^k,X_{\tau_n(T_i^k)-}^{n,M})\vert+\vert\partial_{x}\widetilde{c}(\tau_n(T_i^k),Z_i^k,X_{\tau_n(T_i^k)-}^{n,M})\vert)(\vert LZ_i^k\vert_{l}+\vert LX_{\tau_n(T_i^k)-}^{n,M}\vert_{l})\\
&&+C_{l,p}(T)\sup\limits_{2\leq\vert\beta_1+\beta_2\vert\leq l+2}(\vert\partial_{z}^{\beta_2}\partial_{x}^{\beta_1}\widetilde{c}(\tau_n(T_i^k),Z_i^k,X_{\tau_n(T_i^k)-}^{n,M})\vert)\\
&&\times(1+\vert Z_i^k\vert_{l+1}^{l+2}+\vert X_{\tau_n(T_i^k)-}^{n,M}\vert_{l+1}^{l+2})(1+\vert LZ_i^k\vert_{l-1}+\vert LX_{\tau_n(T_i^k)-}^{n,M}\vert_{l-1})\\
&&\leq C_{l,p}(T){\bar{c}}(Z_i^k)(1+\vert LZ_i^k\vert_{l}+\vert LX_{\tau_n(T_i^k)-}^{n,M}\vert_{l}+\vert X_{\tau_n(T_i^k)-}^{n,M}\vert_{l+1}^{l+2}+\vert X_{\tau_n(T_i^k)-}^{n,M}\vert_{l+1}^{l+2}\times(\vert LZ_i^k\vert_{l-1}+\vert LX_{\tau_n(T_i^k)-}^{n,M}\vert_{l-1})).
\end{eqnarray*}\\
Then 
\begin{eqnarray*}
&&B_3\leq\mathbb{E}(\sum_{k=1}^{M-1}\sum_{i=1}^{J_t^k}\vert L\widetilde{c}(\tau_n(T_i^k),Z_i^k,X^{n,M}_{\tau_n(T_i^k)-})\vert_{l})^p  \\
&&\leq C_{l,p}(T)\mathbb{E}\vert\sum_{k=1}^{M-1}\sum_{i=1}^{J_t^k}{\bar{c}}(Z_i^k)(1+\vert LZ_i^k\vert_{l}+\vert LX_{\tau_n(T_i^k)-}^{n,M}\vert_{l}+\vert X_{\tau_n(T_i^k)-}^{n,M}\vert_{l+1}^{l+2}  \\
&&+\vert X_{\tau_n(T_i^k)-}^{n,M}\vert_{l+1}^{l+2}\times(\vert LZ_i^k\vert_{l-1}+\vert LX_{\tau_n(T_i^k)-}^{n,M}\vert_{l-1}))\vert^p  \\
&&\leq C_{l,p}(T)(B_{3,1}+B_{3,2}+B_{3,3}),
\end{eqnarray*}where 
\[
B_{3,1}=\mathbb{E}(\sum_{k=1}^{M-1}\sum_{i=1}^{J_t^k}{\bar{c}}(Z_i^k)\vert LX_{\tau_n(T_i^k)-}^{n,M}\vert_{l})^p,
\]
\[
B_{3,2}=\mathbb{E}\vert\sum_{k=1}^{M-1}\sum_{i=1}^{J_t^k}{\bar{c}}(Z_i^k)(\vert LZ_i^k\vert_{l}+\vert X_{\tau_n(T_i^k)-}^{n,M}\vert_{l+1}^{l+2}\times\vert LZ_i^k\vert_{l-1})\vert^p,
\]
\[
B_{3,3}=\mathbb{E}\vert\sum_{k=1}^{M-1}\sum_{i=1}^{J_t^k}{\bar{c}}(Z_i^k)(1+\vert X_{\tau_n(T_i^k)-}^{n,M}\vert_{l+1}^{l+2}+\vert X_{\tau_n(T_i^k)-}^{n,M}\vert_{l+1}^{l+2}\times\vert LX_{\tau_n(T_i^k)-}^{n,M}\vert_{l-1})\vert^p.
\]

By (\ref{Burk}),
\begin{eqnarray}
B_{3,1}&=&\mathbb{E}\vert\int_0^t\int_{[1,M)}{\bar{c}}(z)\vert LX_{\tau_n(r)-}^{n,M}\vert_{l}N_\nu(dr,dz)\vert^p  \nonumber\\
&\leq& C_{l,p}(T)\int_0^t\mathbb{E}\vert LX_{\tau_n(r)-}^{n,M}\vert_{l}^pdr.  \label{b31}
\end{eqnarray}

Using Schwartz's inequality, (\ref{Burk}) and (\ref{EuD}), we have \begin{eqnarray*}
&&\mathbb{E}(\sum_{k=1}^{M-1}\sum_{i=1}^{J_t^k}{\bar{c}}(Z_i^k)\vert X_{\tau_n(T_i^k)-}^{n,M}\vert_{l+1}^{l+2}\times\vert LZ_i^k\vert_{l-1})^p\\
&&\leq [\mathbb{E}(\sum_{k=1}^{M-1}\sum_{i=1}^{J_t^k}{\bar{c}}(Z_i^k)\vert X_{\tau_n(T_i^k)-}^{n,M}\vert_{l+1}^{2(l+2)})^p]^{\frac{1}{2}}\times[\mathbb{E}(\sum_{k=1}^{M-1}\sum_{i=1}^{J_t^k}{\bar{c}}(Z_i^k)\vert LZ_i^k\vert_{l-1}^{2})^p]^{\frac{1}{2}}\\
&&= [\mathbb{E}\vert\int_0^t\int_{[1,M)}{\bar{c}}(z)\vert X_{\tau_n(r)-}^{n,M}\vert_{l+1}^{2(l+2)}N_\nu(dr,dz)\vert^p]^{\frac{1}{2}}\times[\mathbb{E}(\sum_{k=1}^{M-1}\sum_{i=1}^{J_t^k}{\bar{c}}(Z_i^k)\vert LZ_i^k\vert_{l-1}^{2})^p]^{\frac{1}{2}}\\
&&\leq C_{l,p}(T)[\mathbb{E}(\sum_{k=1}^{M-1}\sum_{i=1}^{J_t^k}{\bar{c}}(Z_i^k)\vert LZ_i^k\vert_{l-1}^{2})^p]^{\frac{1}{2}}. 
\end{eqnarray*}
Then applying \textbf{Lemma 5.4}, we get
\begin{equation}
    B_{3,2}\leq {C}_{l,p}(T). \label{b32}
\end{equation}

By (\ref{Burk}), (\ref{EuD}) and recurrence hypothesis, we have
\begin{eqnarray}
B_{3,3}&=&\mathbb{E}\vert\int_0^t\int_{[1,M)}{\bar{c}}(z)(1+\vert X_{\tau_n(r)-}^{n,M}\vert_{l+1}^{l+2}+\vert X_{\tau_n(r)-}^{n,M}\vert_{l+1}^{l+2}\times\vert LX_{\tau_n(r)-}^{n,M}\vert_{l-1})N_\nu(dr,dz)\vert^p\nonumber\\
&\leq& {C}_{l,p}(T).  \label{b33}
\end{eqnarray}

So by (\ref{b31}),(\ref{b32}) and (\ref{b33}), 
\begin{equation}
    B_3\leq {C}_{l,p}(T)[1+\int_0^t\mathbb{E}\vert LX_{\tau_n(r)-}^{n,M}\vert_{l}^pdr]. \label{B3}
\end{equation}
Then combining (\ref{B1}),(\ref{B2}) and (\ref{B3}), 
\begin{eqnarray}
\mathbb{E}\vert LX_t^{n,M}\vert_{l}^p\leq {C}_{l,p}(T)[1+\int_0^t\mathbb{E}\vert LX_{\tau_n(r)}^{n,M}\vert_{l}^pdr],  \label{gronL}
\end{eqnarray}
Using Gronwall's lemma for (\ref{gronL}) as for (\ref{gronD}), we conclude that 
\begin{eqnarray}
\sup_{n,M}\mathbb{E}\vert LX_t^{n,M}\vert_{l}^p\leq C_{l,p}(T). \label{L*}
\end{eqnarray}

As a summary of the recurrence argument, we remark that the uniform bound in $n,M$ of the operator $L$ for $l=0$ is due to the \textbf{Hypothesis 2.1,2.4} and \textbf{Lemma 5.3} $c)$, and it propagates to larger $l$ thanks to \textbf{Lemma 5.3} $b)$.
\end{proof}

\bigskip

\noindent\textit{Proof of \textbf{Lemma 4.1}.}

By \textbf{Lemma 5.1} and \textbf{Lemma 5.5}, as a consequence of \textbf{Lemma 3.2}, we have $X_t^M\in\mathcal{D}_{l,p}$ and $\sup\limits_M\Vert X_t^M\Vert_{L,l,p}\leq C_{l,p}(T)$.\qed

\subsection{Proof of Lemma 4.2}
In the following, we turn to the non-degeneracy of $X_t^M$.
We consider the approximate equation (\ref{3.02})
\[
X_{t}^{M }=x+\int_{0}^{t}\int_{[1,M)}\widetilde{c}(r,z,X_{r-}^{M })N_\nu(dr,dz)+\int_{0}^{t}b
_{M}(r,X_{r}^{M })dr+\int_{0}^{t}\int_{\{z\geq M\}}\widetilde{c}
(r,z,X_{r}^{M })W_\nu(dr,dz).
\]%
We can calculate the Malliavin derivatives of the Euler scheme and then by passing to the limit, we have 
\begin{eqnarray}
&&D^{Z}_{(k,i)}X_{t}^{M}=\mathbbm{1}_{\{k\leq M-1\}}\mathbbm{1}_{\{i\leq J_t^k\}}\xi ^{k}_{i}\partial
_{z}\widetilde{c}(T_i^k,Z_i^k,X_{T_i^k-}^M)+\int_{T_{i}^{k}}^{t}\int_{[1,M)}\partial
_{x}\widetilde{c}(r,z,X^M_{r-})D^{Z}_{(k,i)}X_{r-}^{M}N_{\nu }(dr,dz)  \nonumber\\
&&+\int_{T_{i}^{k}}^{t}\partial_xb
_{M}(r,X_{r}^{M })D^{Z}_{(k,i)}X_{s}^{M}dr+\int_{T_{i}^{k}}^{t}\int_{\{z\geq M\}}\partial_x\widetilde{c}(r,z,X_{r}^{M })D^{Z}_{(k,i)}X_{s}^{M}W_\nu(dr,dz).  \label{cM1}\\
&&D^{W}_{(s,z_0)}X_{t}^{M}=\int_{s}^{t}\int_{[1,M)}\partial
_{x}\widetilde{c}(r,z,X^M_{r-})D^{W}_{(s,z_0)}X_{r-}^{M}N_{\nu }(dr,dz)+\int_{s}^{t}\partial_xb
_{M}(r,X_{r}^{M })D^{W}_{(s,z_0)}X_{r}^{M}dr  \nonumber\\
&&+\mathbbm{1}_{\{s\leq t\}}\mathbbm{1}_{\{ z_0\geq M\}}\widetilde{c}(s,z_0,X_{s}^{M })+\int_{s}^{t}\int_{\{z\geq M\}}\partial_x\widetilde{c}(r,z,X_{r}^{M })D^{W}_{(s,z_0)}X_{r}^{M}W_\nu(dr,dz).  \label{cM1.1}
\end{eqnarray}%

We obtain now some explicit expressions for the Malliavin derivatives. We consider the
tangent flow $(Y^M_t)_{t\in[0,T]}$ which is the solution of the linear equation
\[
Y_{t}^{M}=1+\int_{0}^{t}\int_{[1,M)}\partial _{x}\widetilde{c}(r,z,X^M_{r-})Y_{r-}^{M}N_{\nu
}(dr,dz)+\int_{0}^{t}\partial_xb
_{M}(r,X_{r}^{M })Y_{r}^{M}dr+\int_{0}^{t}\int_{\{z\geq M\}}\partial_x\widetilde{c}(r,z,X_{r}^{M })Y_{r}^{M}W_\nu(dr,dz).
\]%
And using It$\hat{o}$'s formula, $\overline{Y}_{t}^{M}=1/Y^M_t$ verifies the equation
\begin{eqnarray*}
&&\overline{Y}_{t}^{M}=1-\int_{0}^{t}\int_{[1,M)}\partial _{x}\widetilde{c}(r,z,X^M_{r-})(1+\partial
_{x}\widetilde{c}(r,z,X^M_{r-}))^{-1}\overline{Y}_{r-}^{M}N_{\nu }(dr,dz)-\int_{0}^{t}\partial_xb
_{M}(r,X_{r}^{M })\overline{Y}_{r}^{M}dr\\
&&-\int_{0}^{t}\int_{\{z\geq M\}}\partial_x\widetilde{c}(r,z,X_{r}^{M })\overline{Y}_{r}^{M}W_\nu(dr,dz)+\frac{1}{2}\int_{0}^{t}\int_{\{z\geq M\}}\vert\partial_x\widetilde{c}(r,z,X_{r}^{M })\vert^2\overline{Y}_{r}^{M}\nu(dz)dr .
\end{eqnarray*}

\bigskip

Applying \textbf{Hypothesis 2.1} with $q^\ast\geq1$ and \textbf{Hypothesis 2.2}, with $K_p$ a constant only depending on $p$, one also has (the proof is standard)
\begin{equation}
\mathbb{E}(\sup_{s\leq t}(\left\vert Y_{s}^{M}\right\vert ^{p}+\left\vert \overline{Y}%
_{s}^{M}\right\vert ^{p}))\leq K_p<\infty.   \label{cM2}
\end{equation}%
\begin{remark}
 Due to (\ref{cbar}), we have 
\[
	\max\Big\{\int_{[1,M)}\vert\bar{c}(z)\vert^p\nu(dz),\int_{[M,\infty)}\vert\bar{c}(z)\vert^p\nu(dz)\Big\}
	\leq \int_{[1,\infty)}\vert\bar{c}(z)\vert^p\nu(dz)= \bar{c}_p,
\]
so the constant in (\ref{cM2}) is uniform with respect to $M$.
\end{remark}

Then using the uniqueness of solution to the equation (\ref{cM1}) and (\ref{cM1.1}), one
obtains 
\begin{eqnarray}
	D^{Z}_{(k,i)}X_{t}^{M}&=&\mathbbm{1}_{\{k\leq M-1\}}\mathbbm{1}_{\{i\leq J_t^k\}}\xi ^{k}_{i}Y_{t}^{M}\overline{Y}^{M}_{T^{k}_{i}-}\partial_{z}\widetilde{c}(T^{k}_{i},Z^{k}_{i},X^M_{T^{k}_{i}-}),\nonumber\\
	 D^{W}_{(s,z_0)}X_{t}^{M}&=&\mathbbm{1}_{\{s\leq t\}}\mathbbm{1}_{\{ z_0\geq M\}}Y_{t}^{M}\overline{Y}^{M}_{s}\widetilde{c}(s,z_0,X_{s}^{M}).  \label{cM3}
\end{eqnarray}%
And the Malliavin covariance of $%
X_{t}^{M}$ is %
\begin{eqnarray}
\sigma _{X_{t}^{M}}=\left\langle DX_{t}^{M},DX_{t}^{M}\right\rangle
_{\mathcal{H}}=\sum_{k=1}^{M-1 }\sum_{i=1}^{J_{t}^{k}}\vert D^{Z}_{(k,i)}X_{t}^{M}\vert^2+\int_0^T\int_{\{z\geq M\}}\vert D^{W}_{(s,z)}X_{t}^{M}\vert^2\nu(dz)ds .  \label{sigmaM}
\end{eqnarray}

\bigskip

In the following, we denote $\lambda_{t}^{M}= \sigma_{X_t^M}$. So the aim is to prove that for every $p\geq 1$,%
\begin{equation}
\mathbb{E}(\vert\lambda _{t}^{M}\vert^{-p})\leq C_p.  \label{cM5}
\end{equation}

We proceed in 5 steps.

\textbf{Step 1} We notice that by (\ref{cM3}) and (\ref{sigmaM}) 

\begin{eqnarray*}
\lambda _{t}^{M} =\sum_{k=1}^{M-1}\sum_{i=1}^{J_{t}^{k}}\xi_i^k \vert Y_{t}^{M}\vert^2\vert\overline{Y}%
_{T_{i}^{k}-}^{M}\vert^2\vert\partial
_{z}\widetilde{c}(T_{i}^{k},Z_{i}^{k},X_{T_{i}^{k}-}^{M
})\vert^{2}+\vert Y_{t}^M\vert^{2}\int_{0}^{t}\vert\overline{Y}_{s}^{M}\vert^2\int_{\{z\geq M\}}\vert\widetilde{c}(s,z,X_{s}^{M })\vert^2\nu(dz)ds.
\end{eqnarray*}%
We recall the ellipticity hypothesis (\textbf{Hypothesis 2.3}): There exists a function $\underline{c}(z)$ such that 
\[
\left\vert \partial _{z}\widetilde{c}(s,z,x)\right\vert ^{2}\geq \underline{c}(z)\quad and\quad \left\vert \widetilde{c}(s,z,x)\right\vert ^{2}\geq \underline{c}(z).
\]%
In particular 
\[
\int_{\{z\geq M\}}\vert \widetilde{c}(s,z,x)\vert^2\nu
(dz)\geq \int_{\{z \geq M\}}\underline{c}(z)\nu (dz),
\]%
so that%
\[
\lambda _{t}^{M}\geq Q_{t}^{-2}\times (\sum_{k=1}^{M-1}\sum_{i=1}^{J_{t}^{k}}%
\xi_i^k\underline{c}(Z_{i}^{k})+t\int_{\{z \geq M\}}%
\underline{c}(z)\nu (dz))\quad with\quad Q_{t}=\inf_{s\leq t}\vert Y_{s}^M%
\overline{Y}_{t}^M\vert.
\]%
We denote 
\[
\rho _{t}^{M}=\sum_{k=1}^{M-1}\sum_{i=1}^{J_{t}^{k}}\xi_i^k\underline{c}(Z_{i}^{k}),\quad \bar{\rho} _{t}^{M}=\sum_{k=M}^{\infty}\sum_{i=1}^{J_{t}^{k}}\xi_i^k\underline{c}(Z_{i}^{k}),\quad \alpha ^{M}=\int_{\{z\geq
M\}}\underline{c}(z)\nu (dz).
\]%
By (\ref{cM2}), $(\mathbb{E}\sup\limits_{s\leq
t}\left\vert Y_{s}^M\overline{Y}_{t}^M\right\vert ^{4p})^{1/2}\leq C<\infty,$ so that 
\begin{eqnarray}
\mathbb{E}(\vert \lambda_t^M\vert^{-p})\leq C(\mathbb{E}(\vert \rho _{t}^{M}+t\alpha ^{M}\vert^{-2p}))^{\frac{1}{2}}. \label{Step1}
\end{eqnarray}

\textbf{Step 2 }Let $\Gamma(p)=\int_0^\infty s^{p-1}e^{-s}ds$. By a change of variables, we have the numerical
equality 
\[
\frac{1}{(\rho _{t}^{M}+t\alpha ^{M})^{p}}=\frac{1}{\Gamma (p)}\int_{0}^{\infty
}s^{p-1}e^{-s(\rho _{t}^{M}+t\alpha ^{M})}ds 
\]%
which, by taking expectation, gives 
\begin{eqnarray}
\mathbb{E}(\frac{1}{(\rho _{t}^{M}+t\alpha ^{M})^{p}})=\frac{1}{\Gamma (p)}\int_{0}^{\infty
}s^{p-1}\mathbb{E}(e^{-s(\rho _{t}^{M}+t\alpha ^{M})})ds.  \label{Gamma}
\end{eqnarray}

\textbf{Step 3 (splitting).} In order to compute $\mathbb{E}(e^{-s(\rho _{t}^{M}+t\alpha ^{M})})$ we have
to interpret $\rho _{t}^M$ in terms of Poisson measures. We recall that we suppose the "splitting hypothesis" (\ref{3.13}):
\[
\mathbbm{1}_{I_k}(z)\frac{\nu(dz)}{m_k}\geq \mathbbm{1}_{I_k}(z)\varepsilon _{k}dz,
\]with $I_k=[k,k+1),\ m_k=\nu(I_k)$.
We also have the function $\psi $ and $m(\psi )=\int_{\mathbb{R}} \psi
(t)dt.$ And we use the basic decomposition 
\[
Z_{i}^{k}= \xi _{i}^{k}V_{i}^{k}+(1-\xi _{i}^{k})U_{i}^{k}
\]%
where $V_{i}^{k},U_{i}^{k},\xi _{i}^{k}, k,i\in\mathbb{N}$ are some independent random variables with laws given in (\ref{split}).

For every $k$ we consider a Poisson point measure $N_{k}(ds,d\xi ,dv,du)$
with $\xi \in \{0,1\},v,u\in [1,\infty),s\in[0,T]$ with compensator%
\begin{eqnarray*}
\widehat{N_{k}}(ds,d\xi ,dv,du)&=&\widehat{M}_{k}(d\xi ,dv,du)\times ds  \\
with\quad \widehat{M}_{k}(d\xi ,dv,du) &=& b_k(d\xi )\times \mathbbm{1}_{I_k}(v)\frac{1}{m(\psi )}\psi (v-(k+\frac{1}{2}))dv \\
&\times& \frac{1}{1-\varepsilon _{k}m(\psi )}\mathbbm{1}_{I_k}(u)(\mathbb{P}(Z_{1}^{k}\in
du)-\varepsilon _{k}\psi (u-(k+\frac{1}{2}))du).
\end{eqnarray*}%
Here $b_k(d\xi )$ is the Bernoulli law of parameter $\varepsilon
_{k}m(\psi )$. The intervals $I_k,k\in\mathbb{N}$ are disjoint so the Poisson point measures $N_{k},k=1,\cdots,M-1$ are independent. Then 
\[
\sum_{i=1}^{J_{t}^{k}}\xi _{i}^{k}{\underline{c}}(Z_{i}^{k})=%
\sum_{i=1}^{J_{t}^{k}}\xi _{i}^{k}{\underline{c}}(\xi _{i}^{k}V_{i}^{k}+(1-\xi _{i}^{k})U_{i}^{k})=\int_{0}^{t}\int_{\{0,1\}}\int_{[1,\infty)^2} \xi {\underline{c}}(\xi v+(1-\xi
)u)N_{k}(ds,d\xi ,dv,du).
\]%
In order to get compact notation, we put together all the measures $%
N_{k},k\leq M-1.$ Since they are independent we get a new Poisson point measure
that we denote by $\Theta.$ And we have 
\begin{eqnarray*}
\rho _{t}^M &=&\sum_{k=1}^{M-1}\sum_{i=1}^{J_{t}^{k}}\xi _{i}^{k}%
{\underline{c}}(Z_{i}^{k})=\int_{0}^{t}\int_{\{0,1\}}\int_{[1,\infty)^2} \xi {\underline{c}}(\xi v+(1-\xi )v)\Theta(ds,d\xi ,dv,du).
\end{eqnarray*}

\textbf{Step 4 }Using It\^{o}'s formula,

\begin{eqnarray*}
\mathbb{E}(e^{-s\rho _{t}^M}) &=&1+\mathbb{E}\int_{0}^{t}\int_{\{0,1\}}\int_{[1,\infty)^2} (e^{-s(\rho _{r-}^M+\xi {\underline{c}}%
(\xi v+(1-\xi )v))}-e^{-s\rho _{r-}^M})\widehat{\Theta}(dr,d\xi ,dv,du) \\
&=&1-\int_{0}^{t}\mathbb{E}(e^{-s\rho _{r-}^M})dr\int_{\{0,1\}}\int_{[1,\infty)^2} (1-e^{-s\xi {\underline{c}}(\xi
v+(1-\xi )v)})\sum_{k=1}^{M-1 }\widehat{M}_{k}(d\xi ,dv,du).
\end{eqnarray*}%
Solving the above equation we obtain 
\begin{eqnarray*}
\mathbb{E}(e^{-s\rho _{t}^M}) &=&\exp (-t\sum_{k=1}^{M-1 }\int_{\{0,1\}}\int_{[1,\infty)^2} (1-e^{-s\xi {\underline{c}}(\xi v+(1-\xi
)u)})\widehat{M}_{k}(d\xi ,dv,du)).
\end{eqnarray*}
We compute
\begin{eqnarray*}
\int_{\{0,1\}\times [1,\infty)^2}(1-e^{-s\xi {\underline{c}}(\xi
v+(1-\xi )u)})\widehat{M}_{k}(d\xi ,dv,du) = \varepsilon _{k}m(\psi )\int_{k}^{k+1}(1-e^{-s{\underline{c}}(v)})\frac{1}{m(\psi )}\psi (v-(k+\frac{1}{2}))dv.
\end{eqnarray*}%
Since $\psi \geq 0$ and $\psi (z)=1$ if $\left\vert z\right\vert \leq \frac{1%
}{4}$ it follows that the above term is larger than 
\[
\varepsilon _{k}\int_{k+\frac{1}{4}}^{k+\frac{3}{4}}(1-e^{-s{\underline{c}}%
(v)})dv.
\]%
Finally this gives 
\begin{eqnarray*}
\mathbb{E}(e^{-s\rho _{t}^M}) &\leq &\exp (-t\sum_{k=1}^{M-1 }\varepsilon
_{k}\int_{k+\frac{1}{4}}^{k+\frac{3}{4}}(1-e^{-s{\underline{c}}(v)})dv) \\
&=&\exp (-t\int_{1}^{M }(1-e^{-s{\underline{c}}(v)})m(dv)),
\end{eqnarray*}%
 with 
\begin{eqnarray}
m(dv)=\sum_{k=1}^{\infty }\varepsilon _{k}1_{(k+\frac{1}{4},k+\frac{3}{4}%
)}(v)dv. \label{mm}
\end{eqnarray}

In the same way, we get
\[
\mathbb{E}(e^{-s\bar{\rho} _{t}^{M}})\leq \exp (-t\int_{M}^{\infty }(1-e^{-s\underline{{c}}%
(v)})m(dv)).
\]
Notice that $t\alpha^M\geq\mathbb{E}(\bar{\rho}_t^M)$. Then using Jensen's inequality for the convex function $f(x)=e^{-sx},s,x>0$, we have 
\begin{eqnarray*}
e^{-st\alpha ^{M}}\leq e^{-s\mathbb{E}\bar{\rho} _{t}^{M}}\leq \mathbb{E}(e^{-s\bar{\rho} _{t}^{M}})\leq \exp (-t\int_{M}^{\infty }(1-e^{-s\underline{{c}}%
(v)})m(dv)).
\end{eqnarray*}
So for every $M\in\mathbb{N}$, we get
\begin{eqnarray}
\mathbb{E}(e^{-s(\rho _{t}^{M}+t\alpha ^{M})}) &=&e^{-st\alpha ^{M}}\times
\mathbb{E}(e^{-s\rho _{t}^{M}}) \nonumber\\
&\leq &\exp (-t\int_{M}^{\infty}(1-e^{-s\underline{{c}}%
(v)})m(dv))\times \exp (-t\int_{1}^{M}(1-e^{-s\underline{{c}}%
(v)})m(dv))  \nonumber\\
&=&\exp (-t\int_{1}^{\infty}(1-e^{-s\underline{{c}}%
(v)})m(dv)),  \label{sew}
\end{eqnarray}and the last term does not depend on $M$.

\bigskip

Now we will use the Lemma 14 from $\cite{ref4}$, which states the following.
\begin{lemma}
We consider an abstract measurable space $E$, a $\sigma$-finite measure $\eta$ on this space and a non-negative measurable function $f:E\rightarrow\mathbb{R}_+$ such that $\int_Efd\eta<\infty.$ For $t>0$ and $p\geq1$, we note
\[
\alpha_f(t)=\int_E(1-e^{-tf(a)})\eta(da)\quad and\quad I_t^p(f)=\int_0^\infty s^{p-1}e^{-t\alpha_f(s)}ds.
\]
We suppose that for some $t>0$ and $p\geq1$, 
\begin{equation}
\underline{\lim }_{u\rightarrow \infty }\frac{1}{\ln u}\eta(f\geq 
\frac{1}{u})>p/t,  \label{cm}
\end{equation}%
then $I_t^p(f)<\infty.$
\end{lemma}
\bigskip

We will use the above lemma for $\eta=m$ and $f=\underline{c}$. So if we have 
\begin{equation}
\underline{\lim }_{u\rightarrow \infty }\frac{1}{\ln u}m({\underline{c}}\geq 
\frac{1}{u})=\infty ,  \label{cm6}
\end{equation}%
then for every $p\geq 1,t>0,M\geq1$, (\ref{Gamma}),(\ref{sew}) and \textbf{Lemma 5.6} give
\begin{eqnarray}
\mathbb{E}(\frac{1}{\rho _{t}^{M}+t\alpha ^{M}})^{2p} &=&\frac{1}{\Gamma (2p)}\int_{0}^{\infty
}s^{2p-1}\mathbb{E}(e^{-s(\rho _{t}^{M}+t\alpha ^{M})})ds  \label{cm7} \\
&\leq &\frac{1}{\Gamma (2p)}\int_{0}^{\infty }s^{2p-1}\exp (-t\int_{1}^{\infty
}(1-e^{-s{\underline{c}}(v)})m(dv))ds<\infty .  \nonumber
\end{eqnarray}
Finally using (\ref{Step1}), we conclude that if (\ref{cm6}) holds, then
\begin{eqnarray}
\sup_{M}\mathbb{E}(\lambda_t^M)^{-p}<\infty. \label{detXt}
\end{eqnarray}

\textbf{Step 5 }Now the only problem left is to compute $m({\underline{c}}\geq \frac{1}{u}).$ It seems
difficult to discuss this in a completely abstract framework. So we suppose \textbf{Hypothesis 2.4} (a): There exists a constant $\varepsilon_{\ast}>0$ and there are some $\alpha_1>\alpha_2>0$ such that for every $k\in\mathbb{N}$,
\begin{equation*}
\mathbbm{1}_{I_k}(z)\frac{\nu (dz)}{m_k}\geq  \mathbbm{1}_{I_k}(z)\varepsilon_kdz \quad with\quad \varepsilon _{k}=\frac{\varepsilon_{\ast}}{{(k+1)}^{1-{{\alpha}} }},\ for\ any\ {\alpha}\in(\alpha_2,\alpha_1],\quad
and\quad
{\underline{c}}(z)\geq e^{- z^{\alpha_2 }},
\end{equation*}%
Then $\{z:{\underline{c}}(z)\geq \frac{1}{u}\}\supseteq\{z:(\ln u)^{1/{\alpha_2} }\geq z\}.$ In particular, for $k\leq\lfloor(\ln u)^{1/\alpha_2}\rfloor-1:=k(u)$, one has $I_k\subseteq\{z:\underline{c}(z)\geq\frac{1}{u}\}$.
Then for $u$ large enough, we compute%
\begin{eqnarray*}
m({\underline{c}} \geq \frac{1}{u})&\geq&\sum_{k=1}^{k(u)}m(I_k)\geq \frac{1}{2}\sum_{k=1}^{k(u)}\varepsilon _{k}\geq \frac{1}{2}\varepsilon_{\ast}\sum_{k=1}^{k(u)}\frac{1}{%
(k+1)^{1-{{\alpha}} }}\geq \frac{1}{2}\varepsilon_{\ast}\int_{2}^{(\ln u)^{1/{\alpha_2} }}\frac{1}{%
z^{1-{{\alpha}} }}dz \\
&=&\frac{\varepsilon_{\ast}}{{2{\alpha}} }((\ln u)^{{\alpha} /{\alpha_2} }-2^{{{\alpha}}  }).
\end{eqnarray*}%
Since ${\alpha}>\alpha_2$, (\ref{cm6}) is verified and we obtain (\ref{detXt}).

\bigskip
\bigskip

Now we consider \textbf{Hypothesis 2.4} (b): We suppose that there exists a constant $\varepsilon_{\ast}>0$ and there are some $\alpha >0$ such that for every $k\in\mathbb{N}$,
\begin{equation*}
\mathbbm{1}_{I_k}(z)\frac{\nu (dz)}{m_k}\geq  \mathbbm{1}_{I_k}(z)\varepsilon_kdz\quad with\quad \varepsilon _{k}=\frac{\varepsilon_{\ast}}{k+1},\quad and\quad {\underline{c}}(z)\geq \frac{1}{%
 z^{\alpha }}.
\end{equation*}%
Now $\{z:{\underline{c}}(z)\geq \frac{1}{u}\} \supseteq \{z:z\leq u^{1/\alpha }\}
$. Then for $u$ large enough,
\[
m({\underline{c}}\geq \frac{1}{u})\geq \frac{1}{2}\varepsilon_{\ast}\sum_{k=1}^{\lfloor u^{1/\alpha }\rfloor-1}\frac{1%
}{k+1}\geq \frac{1}{2}\varepsilon_{\ast}\int_{2}^{u^{1/\alpha }}\frac{dz}{z}= \frac{1}{2}\varepsilon_{\ast}(\frac{1}{\alpha}\ln
u-\ln{2}).
\]%
And consequently 
\[
\underline{\lim }_{u\rightarrow \infty }\frac{1}{\ln u}m({\underline{c}}\geq \frac{1}{u})\geq \frac{\varepsilon_{\ast}}{2\alpha }.
\]%
Using \textbf{Lemma 5.6}, this gives: if 
\[
\frac{2p}{t}< \frac{\varepsilon_{\ast}}{2\alpha }\quad \Leftrightarrow \quad t> \frac{4p \alpha}{\varepsilon_{\ast}}
\]
then  
\[
\sup_M\mathbb{E}(\frac{1}{\rho _{t}^{M}+t\alpha ^{M}})^{2p}<\infty,
\]
and we have $\sup\limits_M\mathbb{E}(\lambda_t^{M})^{-p}<\infty$.\qed

\bigskip

\subsection{Some proofs concerning Section 4.2}
We will prove that the triplet $(\mathcal{S},D,L)$ defined in Section 4.2 is an IbP framework.
Here, we only show that $D^q$ is closable and $L$ verifies the duality formula (\ref{0.01}).
To do so, we  introduce the divergence operator $\delta$.  We denote the space of simple processes by 
\begin{eqnarray*}
\mathcal{P}=\{u=((\bar{u}^k_i)_{\substack{1\leq i\leq m^\prime\\1\leq k\leq m}},\sum_{r=1}^{n}u_r\varphi_r):\bar{u}^k_i,u_r\in\mathcal{S},\varphi_r\in L^2(\mathbb{R}_{+}\times\mathbb{R}_{+},\nu\times Leb),m^\prime,m,n\in\mathbb{N}\} .
\end{eqnarray*}
For $u=((\bar{u}^k_i)_{\substack{1\leq i\leq m^\prime\\1\leq k\leq m}},\sum_{r=1}^{n}u_r\varphi_r)\in\mathcal{P}$, we denote $u^Z=(\bar{u}^k_i)_{\substack{1\leq i\leq m^\prime\\1\leq k\leq m}}$ and $u^W=\sum_{r=1}^{n}u_r\varphi_r$, so that $u=(u^Z,u^W)$.\\
We notice that $\mathcal{P}$ is dense in $L^2(\Omega;\mathcal{H})$, with $\mathcal{H}=l_{2}\otimes L^2(\mathbb{R}_{+}\times\mathbb{R}_{+},\nu\times Leb)$.

Then we define the divergence operator $\delta: \mathcal{P}\rightarrow \mathcal{S}$ by
\begin{eqnarray*}
&&\delta(u)=\delta^Z(u^Z)+\delta^W(u^W)\\
with\ &&\delta^Z(u^Z)=-\sum_{k=1}^m\sum_{i=1}^{m^\prime} (D^{Z}_{(k,i)}\bar{u}^k_i+\xi^k_i \bar{u}^k_i\times \theta _{k}(V^{k}_i))\\
&&\delta^W(u^W)=\sum_{r=1}^{n}u_rW_\nu(\varphi_r)-\sum_{r=1}^{n}
\langle D^Wu_r,\varphi_r\rangle_{L^2(\mathbb{R}_{+}\times\mathbb{R}_{+},\nu\times Leb)}.
\end{eqnarray*}

We will show that $\delta$ satisfies the following duality formula: For every $F\in\mathcal{S},u\in\mathcal{P}$, 
\begin{equation}
    \mathbb{E}\langle DF,u\rangle_{\mathcal{H}}=\mathbb{E}F\delta(u). \label{2.00}
\end{equation}

In fact, if we denote $\hat{V}_{i}^{k
}(x)$ the sequence  $(V_{i_0}^{k_0})_{\substack{1\leq i_0\leq m^\prime\\1\leq k_0\leq m}}$ after replacing $V_{i}^{k}$ by $x$, then for any $m^\prime,m\in\mathbb{N}$,
\begin{eqnarray*}
&&\mathbb{E}\langle D^ZF,u^Z\rangle_{l_2}=\mathbb{E}\sum_{k=1}^m\sum_{i=1}^{m^\prime} D_{(k,i)}^ZF\times \bar{u}^k_i\\
&&=\sum_{k=1}^m\sum_{i=1}^{m^\prime}\mathbb{E}\xi^k_i\partial_{v^k_i}f(\omega ,(V_{i_0}^{k_0})_{\substack{1\leq i_0\leq m^\prime\\1\leq k_0\leq m}},(W_\nu(\varphi_{j}))_{j=1}^{n})\bar{u}^k_i(\omega ,(V_{i_0}^{k_0})_{\substack{1\leq i_0\leq m^\prime\\1\leq k_0\leq m}},(W_\nu(\varphi_{j}))_{j=1}^{n})\\
&&=\sum_{k=1}^m\sum_{i=1}^{m^\prime}\mathbb{E}\int_{\mathbb{R}}\xi^k_i\partial_{v^k_i}f(\omega,\hat{V}_i^k(x),(W_\nu(\varphi_{j}))_{j=1}^{n})\times \bar{u}_k(\omega,\hat{V}_i^k(x),(W_\nu(\varphi_{j}))_{j=1}^{n})\frac{\psi_{k}(x)}{m(\psi)}dx\\
&&=-\sum_{k=1}^m\sum_{i=1}^{m^\prime}\mathbb{E}\int_{\mathbb{R}}\xi^k_if(\omega,\hat{V}_i^k(x),(W_\nu(\varphi_{j}))_{j=1}^{n})\times[\partial_{v^k_i}\bar{u}^k_i(\omega,\hat{V}_i^k(x),(W_\nu(\varphi_{j}))_{j=1}^{n})\\
&&+\bar{u}^k_i(\omega,\hat{V}_i^k(x),(W_\nu(\varphi_{j}))_{j=1}^{n})\frac{\partial_{x}\psi_{k}(x)}{\psi_{k}(x)}]\frac{\psi_{k}(x)}{m(\psi)}dx\\
&&=-\sum_{k=1}^m\sum_{i=1}^{m^\prime}\mathbb{E}F[D^Z_{(k,i)}\bar{u}^k_i+\xi^k_i\bar{u}^k_i\partial_{x}(\ln \psi_{k}(V^k_i))]=\mathbb{E}(F\delta^Z(u^Z)).
\end{eqnarray*}
On the other hand, since $L^2(\mathbb{R}_{+}\times\mathbb{R}_{+},\nu\times Leb)$ is a separable Hilbert space, we can assume without loss of generality that, in the definition of simple functionals, $(\varphi_1,\cdots,\varphi_m,\cdots)$ is the orthogonal basis of the space $L^2(\mathbb{R}_{+}\times\mathbb{R}_{+},\nu\times Leb)$. \\
Then with $p_r=\int_{\mathbb{R}_{+}\times\mathbb{R}_{+}}\varphi_r^2(s,z)\nu(dz)ds$, for any $n\in\mathbb{N}$,
\begin{eqnarray*}
&&\mathbb{E}\langle D^WF,u^W\rangle_{L^2(\mathbb{R}_{+}\times\mathbb{R}_{+},\nu\times Leb)}=\mathbb{E}\int_{\mathbb{R}_{+}\times\mathbb{R}_{+}}D^{W}_{(s,z)}F\times \sum_{r=1}^{n}u_r\varphi_r(s,z)\ \nu(dz)ds\\
&&=\mathbb{E}\sum_{r=1}^{n}\partial_{ w_r}f(\omega ,(V_{i}^k)_{\substack{1\leq i\leq m^\prime\\1\leq k\leq m}},(W_\nu(\varphi_{j}))_{j=1}^{n})u_r(\omega ,(V_{i}^k)_{\substack{1\leq i\leq m^\prime\\1\leq k\leq m}},(W_\nu(\varphi_{j}))_{j=1}^{n})p_r\\
&&=\sum_{r=1}^{n}\mathbb{E}\int_{\mathbb{R}}\partial_{ w_r}f(\omega ,(V_{i}^k)_{\substack{1\leq i\leq m^\prime\\1\leq k\leq m}},W_\nu(\varphi_{1}),\cdots,W_\nu(\varphi_{r-1}),y,W_\nu(\varphi_{r+1}),\cdots,W_\nu(\varphi_{n}))\\
&&\times u_r(\omega ,(V_{i}^k)_{\substack{1\leq i\leq m^\prime\\1\leq k\leq m}},W_\nu(\varphi_{1}),\cdots,W_\nu(\varphi_{r-1}),y,W_\nu(\varphi_{r+1}),\cdots,W_\nu(\varphi_{n}))\frac{1}{\sqrt{2\pi p_r}}e^{-\frac{y^2}{2p_r}} dy\times p_r\\
&&=-\sum_{r=1}^{n}\mathbb{E}\int_{\mathbb{R}}f(\omega ,(V_{i}^k)_{\substack{1\leq i\leq m^\prime\\1\leq k\leq m}},W_\nu(\varphi_{1}),\cdots,W_\nu(\varphi_{r-1}),y,W_\nu(\varphi_{r+1}),\cdots,W_\nu(\varphi_{n}))\\
&&\times[\partial_{w_r}u_r(\omega ,(V_{i}^k)_{\substack{1\leq i\leq m^\prime\\1\leq k\leq m}},W_\nu(\varphi_{1}),\cdots,W_\nu(\varphi_{r-1}),y,W_\nu(\varphi_{r+1}),\cdots,W_\nu(\varphi_{n}))\\
&&-\frac{y}{p_r}u_r(\omega ,(V_{i}^k)_{\substack{1\leq i\leq m^\prime\\1\leq k\leq m}},W_\nu(\varphi_{1}),\cdots,W_\nu(\varphi_{r-1}),y,W_\nu(\varphi_{r+1}),\cdots,W_\nu(\varphi_{n}))]\frac{1}{\sqrt{2\pi p_r}}e^{-\frac{y^2}{2p_r}} dy\times p_r\\
&&=\mathbb{E}F(\sum_{r=1}^{n}u_rW_\nu(\varphi_r)-\sum_{r=1}^{n}
\langle D^Wu_r,\varphi_r\rangle_{L^2(\mathbb{R}_{+}\times\mathbb{R}_{+},\nu\times Leb)})=\mathbb{E}(F\delta^W(u^W)).
\end{eqnarray*}
Then (\ref{2.00}) is proved. Using this duality formula recursively, we can show the closability of $D^q$. If there exists $u\in L^2(\Omega;\mathcal{H}^{\otimes q})$ such that $F_n\rightarrow0$ in $L^2(\Omega)$ and $D^qF_n\rightarrow u$ in  $L^2(\Omega;\mathcal{H}^{\otimes q})$, then for any $h_1,\cdots,h_q\in \mathcal{P}$, $\mathbb{E}\langle u,h_1\otimes\cdots\otimes h_q\rangle_{\mathcal{H}^{\otimes q}}=\lim\limits_{n\rightarrow\infty}\mathbb{E}\langle D^qF_n,h_1\otimes\cdots\otimes h_q\rangle_{\mathcal{H}^{\otimes q}}=\lim\limits_{n\rightarrow\infty}\mathbb{E}F_n\delta(h_1\delta(h_2(\cdots\delta(h_q))))=0$. Since $\mathcal{P}^{\otimes q}$ is dense in $L^2(\Omega;\mathcal{H}^{\otimes q})$, we conclude that $u=0$. This implies that $D^q$ is closable. 

We notice that from the definition of $\delta$ and $L$, we get immediately that $LF=\delta(DF),\ \forall F\in\mathcal{S}$. And if we replace $u$ by $DG$ in (\ref{2.00}) for $G\in\mathcal{S}$, we get the duality formula of $L$ (\ref{0.01}).

\bigskip\bigskip

{\bf Data avaibility statement.} Data sharing is not applicable to this article as no datasets were generated or analyzed during the current study.

\bigskip 

\end{document}